\documentclass[11pt]{article}

\usepackage{latexsym}
\usepackage{amssymb}
\usepackage{amsthm}
\usepackage{amscd}
\usepackage{amsmath}
\usepackage{mathrsfs}
\usepackage{graphicx}
\usepackage{shuffle}
\usepackage{hyperref}
\usepackage{mathtools}
\usepackage[bbgreekl]{mathbbol}
\usepackage{mathdots}
\usepackage{ytableau}
\usepackage{tikz-cd}
\usepackage{enumerate}

\usepackage{stmaryrd}

\usepackage[colorinlistoftodos]{todonotes}
\usetikzlibrary{chains,scopes}
\usetikzlibrary{shapes.geometric,positioning}

\DeclareSymbolFontAlphabet{\mathbb}{AMSb}
\DeclareSymbolFontAlphabet{\mathbbl}{bbold}

\numberwithin{equation}{section}
\allowdisplaybreaks[1]

\theoremstyle{definition}
\newtheorem* {theorem*}{Theorem}
\newtheorem* {conjecture*}{Conjecture}
\newtheorem{theorem}{Theorem}[section]

\theoremstyle{definition}

\newtheorem* {example*}{Example}

\newtheorem{lemma}[theorem]{Lemma}
\theoremstyle{definition}
\newtheorem{definition}[theorem]{Definition}
\theoremstyle{definition}

\newtheorem{conjecture}[theorem]{Conjecture}
\newtheorem{proposition}[theorem]{Proposition}
\newtheorem{corollary}[theorem]{Corollary}

\newtheorem*{remark*}{Remark}
\theoremstyle{definition}
\newtheorem{remark}[theorem]{Remark}
\theoremstyle{definition}
\newtheorem {example}[theorem]{Example}
\theoremstyle{definition}

\theoremstyle{definition}

\theoremstyle{definition}

\def\({\left(}
\def\){\right)}

\newcommand{\sI}{\mathscr{I}}

\newcommand{\sK}{\mathscr{K}}
\newcommand{\CC}{\mathbb{C}}
\newcommand{\QQ}{\mathbb{Q}}
\newcommand{\CP}{\mathbb{P}}

\newcommand{\cK}{\mathcal{K}}
\newcommand{\cO}{\mathcal{O}}

\def\cX{\mathcal{X}}

\def\Hom{\mathrm{Hom}}

\def\Tor{\mathrm{Tor}}
\def\CC{\mathbb{C}}

\def\ZZ{\mathbb{Z}}

\def\Res{\mathrm{Res}}

\def\ch{\operatorname{ch}}
\def\spanning{\textnormal{-span}}

\newcommand{\cM}{\mathcal{M}}
\newcommand{\cN}{\mathcal{N}}

\def\barr{\begin{array}}
\def\earr{\end{array}}
\def\ba{\begin{aligned}}
\def\ea{\end{aligned}}
\def\be{\begin{equation}}
\def\ee{\end{equation}}

\def\qquand{\qquad\text{and}\qquad}
\def\quand{\quad\text{and}\quad}

\newcommand{\I}{\mathcal{I}}

\def\cM{\mathcal M}
\def\DesR{\mathrm{Des}_R}
\def\DesL{\mathrm{Des}_L}

\def\id{\mathrm{id}}

\def\ben{\begin{enumerate}}
\def\een{\end{enumerate}}

\def\fpf{{\textsf {FPF}}}

\def\D{\hat D}

\def\Des{\mathrm{Des}}

\def\FF{\mathbb{F}}

\def\Ifpf{\I_\fpf}

\newcommand{\rank}{\operatorname{rank}}
\renewcommand{\dim}{\operatorname{dim}}

\def\iG{\mathfrak{G}^{\textsf{O}}}
\def\iGconj{\widetilde{\mathfrak{G}}^{\textsf{O}}}

\def\arcstart{\ \xy<0cm,-.15cm>\xymatrix@R=.1cm@C=.3cm }
\newcommand{\arcstartc}[1]{\ \xy<0cm,-.15cm>\xymatrix@R=.1cm@C=#1cm}

\def\ellfpf{\ell_\fpf}

\def\cT{\mathcal{T}}

\def\cF{\mathcal{F}}

\newcommand{\Sym}{\textsf{Sym}}

\def\D{\textsf{D}}

\def\GQ{G\hspace{-0.2mm}Q}

\def\ss{/\hspace{-1mm}/}

\def\PP{\ZZ_{>0}}
\def\NN{\ZZ_{\geq 0}}

\newcommand{\cC}{\mathcal{C}}

\def\ss{/\hspace{-1mm}/}

\def\cU{\mathcal{U}}

\def\mon{\mathsf{mon}}

\def\init{\operatorname{init}}
\def\bei{\begin{itemize}}
\def\eei{\end{itemize}}

\definecolor{darkred}{rgb}{0.7,0,0} 
\newcommand{\defn}[1]{{\color{darkred}\emph{#1}}} 

\def\KK{\mathbb{K}}

\def\B{\mathsf{B}}

\def\cl{\mathrm{cl}}

\def\Mat{\mathsf{Mat}}

\def\ss{\mathrm{ss}}
\def\SSMat{\mathsf{Mat}^{\ss}}

\def\SSI{I^{\ss}}
\def\SSJ{J^{\ss}}
\def\SScM{\cM^{\ss}}
\def\SScX{\cU^{\ss}}
\def\ssx{u^{\ss}}
\def\SScF{\cF^{\ss}}

\def\sym{\mathrm{sym}}
\def\SymMat{\mathsf{Mat}^{\sym}}
\def\SymI{I^{\sym}}
\def\SymJ{J^{\sym}}
\def\SymcM{\cM^{\sym}}
\def\SymcX{\cU^{\sym}}

\def\MX{\mathsf{MSV}}
\def\openMX{\mathsf{MSC}}

\newcommand{\openSSX}{\openMX^{\ss}}
\def\SSX{\MX^{\ss}}

\newcommand{\openSymX}{\openMX^{\sym}}
\def\SymX{\MX^{\sym}}

\DeclareMathOperator{\Spec}{Spec}
\DeclareMathOperator{\Rep}{Rep}
\DeclareMathOperator{\pt}{pt}
\DeclareMathOperator{\wght}{wt}

\newcommand{\ltriang}{\raisebox{-0.5pt}{\tikz{\draw (0,0) -- (.25,0) -- (0,.25) -- (0,0);}}}

\def\adiag{\operatorname{adiag}}
\def\dom{\mathsf{dom}}

\def\DesFPF{\Des_V}
\newcommand{\IPDideal}{\SymJ}

\newcommand{\ltriangeq}{\ltriang}

\def\cF{\mathcal{T}}

\def\sK{\mathscr{P}}

\def\cN{\mathsf{Index}}
\def\colon{:}
\def\ch{\operatorname{\mathsf{char}}}
\def\cK{\mathcal{K}}
\def\cM{\mathcal{M}}

\usepackage{fullpage}

\begin{document}
\title{Ideal transition systems}
\author{
Eric Marberg \\ {\tt eric.marberg@gmail.com}
 \and 
 Brendan Pawlowski \\ {\tt br.pawlowski@gmail.com}
}

\date{}

\maketitle

\begin{abstract}
We study an inductive method of computing initial ideals and Gr\"obner bases for families of ideals  in a polynomial ring. 
This method starts from a given set of pairs $(I,J)$ where $I$ is any ideal and $J$ is a monomial ideal contained in the initial ideal of $I$. These containments become a system of equalities if one can establish a particular transition recurrence among the chosen ideals. We describe explicit constructions of such systems in two motivating cases---namely, for the ideals of matrix Schubert varieties and their skew-symmetric analogues. Despite many formal similarities with these examples, for the symmetric versions of matrix Schubert varieties, it is an open problem to construct the same kind of transition system. We present several conjectures that would follow from such a construction, while also discussing the special obstructions arising in the symmetric case.
\end{abstract}


\section{Introduction}
Fix a positive integer $n$ and let $\Mat_{n\times n}$ be the variety of $n\times n$ matrices over an algebraically closed field $\KK$.
Write $\B_n \subset \Mat_{n \times n}$ for the group of 
 invertible lower triangular matrices in this set. 
 
 Recall that a \defn{partial permutation matrix} is a binary matrix with at most one nonzero entry in each row and column.
 A version of Gaussian elimination shows that every $M \in \Mat_{n \times n}$ can be written as $b_1 w b_2^\top$ for $b_1, b_2 \in \B_n$ and a unique partial permutation matrix $w$. Equivalently, the $n\times n$ partial permutation matrices are representatives for the distinct $\B_n \times \B_n$-orbits on $\Mat_{n \times n}$. These orbits are called \defn{matrix Schubert cells} $\openMX_w$, and their Zariski closures are called \defn{matrix Schubert varieties} $\MX_w$. We can express these sets using rank conditions as
\begin{align} \label{eq:MX-defn}
    \openMX_w &= \{M \in \Mat_{n\times n} : \text{$\rank M_{[i][j]} = \rank w_{[i][j]}$ for all $i,j \in [n]$}\}, \\
    \MX_w &= \{M \in \Mat_{n\times n} : \text{$\rank M_{[i][j]} \leq \rank w_{[i][j]}$ for all $i,j \in [n]$}\}. \nonumber 
\end{align}
Here we write $[n] = \{1,2,\ldots,n\}$, and if $R, C \subseteq [n]$ then $M_{RC}$ denotes the submatrix of $M$ in rows $R$ and columns $C$. Thus, $M_{[i][j]}$ is the upper-left $i \times j$ corner of $M$.

Let $\cX = [x_{ij}]_{i,j\in[n]}$ be a generic matrix of indeterminates and let $\KK[\Mat_{n\times n}] = \KK[x_{ij} : i,j \in [n]]$ be the coordinate ring of $\Mat_{n\times n}$.  For a partial permutation matrix $w$, let $I_w$ be the ideal in $\KK[\Mat_{n\times n}] $ generated by all minors of size $\rank w_{[i][j]} + 1$ in $\cX_{[i][j]}$ for $i,j \in [n]$. 
Since $\rank M \leq r$ if and only if all $(r+1) \times (r+1)$ minors of $M$ vanish, the affine variety $\MX_w$ is exactly the zero set of $I_w$. 

Fulton \cite{FultonEssentialSet} showed that the equations setting these minors to zero also define $\MX_w$ scheme-theoretically, i.e., that $I_w$ is a radical ideal (in fact, it is prime). Knutson and Miller \cite{KnutsonMiller} proved that this generating set of minors is actually a Gr\"obner basis for $I_w$ under an appropriate term order. This fact allows for convenient computation of certain cohomological invariants associated to $I_w$, and can also be used to derive Fulton's result.

In \cite{MP2022}, we reproved these results using a new inductive approach. The key step in this approach was establishing a \defn{transition recurrence} of the form
\begin{equation} 
\textstyle \label{eq:intro-transition}
    I_w + \langle x_{ij} \rangle = \bigcap_{v \in \Phi} I_v
\end{equation}
and using it to deduce properties of $I_w$ from inductively known properties of the ideals $I_v$. 
Here $x_{ij}$ is a particular variable not contained in $I_w$,  $\langle x_{ij} \rangle$ is the principal ideal that $x_{ij}$ generates,
and $\Phi$ is a certain set of partial permutations.

There is a geometric reason for such recurrences to exist. If $w$ and $y$ are two partial permutation matrices, then $\MX_w \cap \MX_y$ is a closed $\B_n$-stable set and therefore decomposes as a union $\bigcup_{v \in \Phi} \MX_v$. A result of Ramanathan \cite{Ramanathan} implies that the scheme $\MX_w \cap \MX_y$ is actually a reduced union of matrix Schubert varieties, so the stronger equality $I_w + I_y = \bigcap_{v \in \Phi} I_v$ also holds. The transition recurrence then follows by an appropriate choice of $y$. 

\begin{example} \label{ex:transition}
    Let $w =w_1w_2w_3= 132 = \left[ \begin{smallmatrix} 1 & 0 & 0 \\ 0 & 0 & 1 \\ 0 & 1 & 0 \end{smallmatrix} \right]$, where we identify $w$ with the permutation matrix having $1$'s in positions $(i,w_i)$. Then every rank condition $\rank M_{[i][j]} \leq \rank w_{[i][j]}$ is vacuous except when $i = j = 2$, giving 
    \begin{equation*}
        \MX_{132} = \left\{A \in \Mat_{3\times 3} : \rank A_{[2][2]} \leq 1\right\} \qquad \text{and} \qquad I_{132} = \langle x_{21}x_{12} - x_{11}x_{22}\rangle.
    \end{equation*}
    The transition recurrence in this case reads as
    \begin{equation*}
        I_{132} + \langle x_{11}\rangle  = \langle x_{11}, x_{21}x_{12}\rangle = \langle x_{11}, x_{21}\rangle \cap \langle x_{11}, x_{12}\rangle = I_{231} \cap I_{312}.
    \end{equation*}
\end{example}

The situation is much the same if we replace $\Mat_{n\times n}$ with the subset $\SSMat_{n\times n}$
of skew-symmetric $n\times n$ matrices over $\KK$, now considering the orbits of the $\B_n$-action $b \cdot M = bMb^\top$. The corresponding $\B_n$-orbits $\openSSX_w$  in $\SSMat_{n\times n}$ and their closures $\SSX_w$ are defined by the same rank conditions as in \eqref{eq:MX-defn}, but now the orbit representatives $w$ are skew-symmetric $\{0,1,-1\}$-matrices with at most one nonzero entry in each row and column.  

 In our previous work \cite{MP2022}, we constructed Gr\"obner bases for the prime ideals of the \defn{skew-symmetric matrix Schubert varieties} $\SSX_w$. These generating sets consist of Pfaffians of principal submatrices of a generic skew-symmetric matrix, and are more complicated than the minors generating the ideals $I_w$ above. However, the proof technique based on transition recurrences is the same as in the classical case.
It would be interesting to know if there is some geometric result similar to \cite{Ramanathan} that explains these skew-symmetric transition recurrences; at present, this is an open question.

The first, partially expository goal of this article is to present a streamlined version of the 
proof techniques just described.
Specifically, we formalize 
an inductive method of computing initial ideals and Gr\"obner bases 
for families of ideals  in a polynomial ring, in terms of what we call \defn{transition systems}. Such a system is composed of a set of pairs $(I,J)$ where 
$I$ is any ideal and $J$ is a monomial ideal  contained in the initial ideal of $I$.
These containments become a system of equalities 
if one can establish transition recurrences analogous to \eqref{eq:intro-transition} among the chosen ideals. 
Section~\ref{ts-sect} develops
the general setup and nontrivial properties of transition systems, following some preliminaries in Section~\ref{prelim-sect}.
Our motivating examples appear in
Sections~\ref{schumat-sect1} and \ref{schumat-sect2}, where we present transition systems 
that compute Gr\"obner bases for the matrix Schubert varieties $\MX_w$ and $\SSX_w$.

Our second goal in this paper is to explain some conjectures related to
 another interesting variant of the matrix Schubert construction, where $\Mat_{n\times n}$ is replaced by the subset of symmetric matrices $\SymMat_{n\times n}$. Using the same action as in the skew-symmetric case, one may consider the $\B_n$-orbits $\openSymX_w$ in $\SymMat_{n\times n}$ and their closures $\SymX_w$,
 which we call \defn{symmetric matrix Schubert varieties}.
 These may again be defined by the rank conditions in \eqref{eq:MX-defn}, but restricted to symmetric matrices.
 When $\KK$ does not have characteristic two (as we assume in Example~\ref{intro-ex} below), the corresponding orbit representatives $w$ are provided by all symmetric $n\times n$ partial permutation matrices. 
 
 Let $\SymcX = [u_{ij}]_{i,j\in[n]}$ be a symmetric matrix of indeterminates with $u_{ij} = u_{ji}$ and write $\KK[\SymMat_{n\times n}] = \KK[u_{ij} : 1 \leq j \leq i \leq n ]$ for the coordinate ring of $\SymMat_{n\times n}$.  
 Then, for any symmetric matrix $w$, we can define an ideal $\SymI_w$ of $\KK[\SymMat_{n\times n}] $ generated by all minors of size $\rank w_{[i][j]} + 1$ in $\SymcX_{[i][j]}$ for $i,j \in [n]$. 
The affine variety $\SymX_w$ is the zero set of $\SymI_w$.

 Now matters become less pleasant: the intersection of two varieties $\SymX_w$ need not be a reduced union of varieties $\SymX_v$, and no transition recurrence exactly like \eqref{eq:intro-transition} holds with the family of ideals $\Sym_w$. The next example demonstrates the sort of bad behavior that can occur. 
 
\begin{example}\label{intro-ex}
The permutation $w=132$ indexes the symmetric matrix Schubert variety
    \begin{equation*}
        \SymX_{132} = \{A \in \SymMat_{3\times 3} : \rank A_{[2][2]} \leq 1\}, \quad \text{which has prime ideal } \SymI_{132} = \langle u_{21}^2 - u_{11}u_{22}\rangle,
    \end{equation*}
    while $\SymX_{213} = \{A \in \SymMat_{3\times 3} : A_{11} = 0\}$ has ideal $\SymI_{213} = \langle u_{11}\rangle$. As a variety, $\SymX_{132} \cap \SymX_{213}$ is just the set of symmetric matrices $\left[ \begin{smallmatrix} 0 & 0 & \ast \\ 0 & \ast & \ast \\ \ast & \ast & \ast \end{smallmatrix} \right]$. However, the scheme structure is non-reduced as
    \begin{equation*}
        \SymI_{132} + \SymI_{213} = \langle u_{21}^2 - u_{11}u_{22}\rangle + \langle u_{11}\rangle = \langle u_{11}, u_{21}^2\rangle.
    \end{equation*}
\end{example}

Because of obstructions like this one, the set of ideals $\SymI_w$ is too small to be assembled 
into a transition system as defined in Section~\ref{ts-sect}.
Nevertheless, we conjecture that a larger transition system exists which includes these ideals.
We explain this conjecture precisely in Section~\ref{schumat-sect3} (see Conjecture~\ref{sym-conj}) and present some evidence that it holds.

There, we also outline some appealing consequences of such a result, for which we have more computational support.
In detail, our main conjecture would imply that each ideal $\SymI_w$ is both radical and prime (Conjecture~\ref{sym-conj4}),
and that the ideal's original generating set of minors is a Gr\"obner basis under an appropriate term order (Conjecture~\ref{sym-conj3b}).
These properties have been checked by computer when $\ch(\KK)=0$ for $n\leq 7$.
Conjecture~\ref{sym-conj} would also imply that $\SymI_w$ is the ideal of the variety $\SymX_w$ (Conjecture~\ref{sym-conj3a}).

In Section~\ref{stable-sect} we study a more involved application of this last conjectural property.
Specifically, assuming Conjecture~\ref{sym-conj3a} as a hypothesis,
we show that the \defn{orthogonal Grothendieck polynomials} introduced in \cite{MP2020} have well-defined stable limits.
This reduces the open problem \cite[Prob.~5.3]{MP2020} to the task of  constructing a transition system that includes all of the ideals $\SymI_w$.

\subsection*{Acknowledgments}

This work was partially supported by Hong Kong RGC grants  16306120 and 16304122. 
We thank 
Darij Grinberg
for useful comments 
and Allen Knutson for helping to clarify our approach in \S 7.

\section{Preliminaries}\label{prelim-sect}

Throughout this article, $\KK$ denotes an arbitrary field that is algebraically closed.

\subsection{Gr\"obner bases}

Fix a positive integer $N$, let 
$x_1,x_2,\dots,x_N$ be commuting variables, and consider the polynomial ring $\KK[{\bf x}] := \KK[x_1,x_2,\dots,x_N]$. 

A \defn{monomial} in $\KK[{\bf x}]$ is an element of the form $x_{1}^{e_1}x_{2}^{e_2}\dots x_{N}^{e_N}$ where  $e_1,e_2,\dots,e_N \geq 0$ are nonnegative integers.
A \defn{monomial ideal} of $\KK[{\bf x}]$ is an ideal generated by a set of monomials.
A \defn{homogeneous} element of $\KK[{\bf x}]$ is a $\KK$-linear combination of monomials of the same degree.
A \defn{homogeneous ideal} of $\KK[{\bf x}]$ is an ideal generated by a set of homogeneous elements.

A \defn{term order} on $\KK[{\bf x}] $ is a total order on the set of  all monomials,
such that $1$ is the unique minimum and such that if $\mon_1$, $\mon_2$, and $\mon_3$ are monomials with $\mon_1 \leq \mon_2$, 
then $\mon_1  \mon_3 \leq \mon_2 \mon_3$. 
Some common examples of term orders are discussed below.

\begin{example}
The \defn{lexicographic term order} on $\KK[{\bf x}]$
has 
$x_1^{a_1}x_2^{a_2}\cdots x_N^{a_N} \leq x_1^{b_1}x_2^{b_2} \cdots x_N^{b_N}$ if and only if
$(a_1,a_2, \ldots, a_N) \leq (b_1, b_2,\ldots, b_N)$ in lexicographic order,
meaning that if there is a smallest index $i$ with $a_i \neq b_i$ then $a_i<b_i$.
\end{example}

\begin{example}
The \defn{(graded) reverse lexicographic term order} on $\KK[{\bf x}]$ has 
\be\label{grlto-eq}
x_1^{a_1}x_2^{a_2}\cdots x_N^{a_N} \leq x_1^{b_1}x_2^{b_2} \cdots x_N^{b_N}
\ee if and only if both 
$ \sum_{i=1}^N a_i \leq \sum_{i=1}^N b_i $ and $ (a_N, \ldots, a_2,a_1) \geq (b_N, \ldots, b_2,b_1)$
in lexicographic order.
\end{example}

\begin{remark}\label{rlto-rem} Later we will consider the reverse lexicographic term order on 
polynomial rings generated by variables $x_{ij}$ indexed by certain pairs $(i,j) \in \ZZ\times \ZZ$.
In this context, we order the indexing pairs $(i,j)$ lexicographically 
to write products of variables in the form \eqref{grlto-eq}.
Then the reverse lexicographic term order
is the unique total order $<$ such that:
\begin{itemize}

\item[(a)] For  monomials of different degrees one has $\mon_1 < \mon_2$  if   $\deg (\mon_1)< \deg (\mon_2)$.

\item[(b)] If $\mon_1 = \prod_{(i,j)} x_{ij}^{a_{ij}}$ and  $\mon_2 = \prod_{(i,j)} x_{ij}^{b_{ij}}$ are distinct with the same degree then one has $\mon_1 < \mon_2$ when  $a_{ij}> b_{ij}$
for the lexicographically largest index $(i,j)$ with $a_{ij}\neq b_{ij}$.
\end{itemize}
Under this term order, if $\mon_1$ and $\mon_2$ are both square-free and of the same degree, then  $\mon_1 < \mon_2$ if and only there is some variable $x_{ij}$ that does not divide both monomials, and the lexicographically largest such variable divides $\mon_1$ but not $\mon_2$.
\end{remark}

\begin{example}Under the reverse lexicographic term order on $\KK[x_{11},x_{12},x_{21},x_{22}]$ one has 
\[ \ba 
1 & < x_{22} < x_{21} < x_{12} < x_{11} \\
& < x_{22}^2 < x_{21}x_{22} < x_{12}x_{22} < x_{11}x_{22} \\
&<  x_{21}^2 <  x_{12}x_{21} < x_{11}x_{21}
\\& < x_{12}^2 < x_{11}x_{12} 
\\& < x_{11}^2 < \cdots .
\ea
\]
\end{example}

Fix a term order and suppose $f = \sum_{\mon} c_\mon \cdot \mon \in \KK[{\bf x}]$ where the sum is over all monomials $\mon \in \KK[{\bf x}]$ and all but finitely many of the coefficients $c_\mon \in \KK$ are zero.
The \defn{initial term} $\init(f)$ of $f$ is either the maximal monomial $\mon$ such that $c_\mon \neq 0$, or zero when $f=0$.

Choose an ideal $I\subseteq \KK[{\bf x}]$. The \defn{initial ideal} of $I$ is
the monomial ideal $\init(I) := \langle\init(f) : f \in I\rangle$.
One can check that this ideal coincides with the vector space $ \KK\spanning\{\init(f) : f \in I\}$, and that  if  $I$ is already a monomial ideal then $\init(I)=I$.
A \defn{Gr\"obner basis} $G$ for $I$, relative to the chosen term order,
is a generating set for $I$ such that $\{ \init(g) :g \in G\}$ generates $\init(I)$.

\begin{example}
If $f \in \KK[{\bf x}]$ is any polynomial and $I \subseteq \KK[{\bf x}]$ is any ideal then
\be
\init(\langle f \rangle) = \langle \init (f) \rangle \quand \init(f I) = \init(f) \init(I).
\ee 
Therefore $G=\{f\}$ is a Gr\"obner basis for the principal ideal $\langle f\rangle$.
\end{example}

\begin{example}
If $I\subsetneq \KK[{\bf x}]$ is a maximal ideal then $I = \langle x_1-c_1,x_2-c_2,\dots,x_N-c_N\rangle$ for certain constants $c_i \in \KK$,
and for any term order $\init(I)$ is the \defn{irrelevant ideal} $\langle x_1,x_2,\dots,x_N\rangle \subsetneq \KK[{\bf x}]$.
In this case the ideal's given generating set $G = \{ x_1-c_1,x_2-c_2,\dots,x_N-c_N\}$ is a Gr\"obner basis. 
\end{example}

Not every generating set for an ideal is a Gr\"obner basis. The ideal $I = \langle x_1,x_2\rangle \subset \KK[x_1,x_2]$ 
is generated by $\{x_2-x_1, x_2\}$, but this is not a Gr\"obner basis for any term order with 
$x_1<x_2$. 
However,   if $I\subseteq \KK[{\bf x}]$
is any ideal then $\init(\init(I)) = \init(I)$.
This means that if $S$ is any set of monomials in $\KK[{\bf x}]$ then $S$ is a Gr\"obner basis for the  ideal  $\langle S\rangle$.

\subsection{More on initial ideals}

We continue the setup of the previous section. There are a few additional properties of initial ideals that will be of use later.
To begin, we recall a few basic facts derived elsewhere. 

    \begin{proposition}[{See \cite[Lem.~2.1.11]{BCRV} and \cite[Prop.~2.4]{MP2022}}] \label{prop:init-facts} Let $I$ and $J$ be ideals in $\KK[{\bf x}]$.   Then:
        \begin{enumerate}
            \item[(a)]  It always holds that $\init(I)+\init(J) \subseteq \init(I+J)$ and $\init(I \cap J) \subseteq \init(I) \cap \init(J)$.
            \item[(b)] Assume $I \subseteq J$. Then $\init(I) \subseteq \init(J)$, and if $\init(I) = \init(J)$ then $I = J$.
            
            \item[(c)] One has $\init(I+J) = \init(I) + \init(J)$ if and only if $\init(I \cap J) = \init(I) \cap \init(J)$. 
        \end{enumerate}
    \end{proposition}

\begin{corollary}\label{cor:init-facts}
Any subset $G$ of an ideal $I$ satisfying $\init(I) \subseteq \langle \init(g) : g \in  G \rangle$
is a Gr\"obner basis.
\end{corollary}

\begin{proof}
In this case  
  $\init(I) \subseteq  \langle \init(g) : g \in  G \rangle\subseteq  \init(\langle G \rangle) \subseteq  \init(I)$ so all containments must be equality.
 Applying Proposition~\ref{prop:init-facts}(b) to $\langle G \rangle \subseteq I$ shows that
$I=\langle G\rangle $ so $G$ is a Gr\"obner basis.
\end{proof}

    \begin{lemma}[{\cite[Lem.~2.6]{MP2022}}]
    \label{lem:hilbert} If $J \subseteq I$ are homogeneous ideals in $\KK[{\bf x}]$ and $f \in \KK[{\bf x}]$ is a non-constant homogeneous polynomial such that $I \cap \langle f\rangle  = fI$ and $I + \langle f\rangle  = J + \langle f\rangle $, then $I = J$. \end{lemma}

The conclusion of this lemma does not hold if the ideals involved are inhomogeneous.
  This can be seen by considering $f = x_1 \in \KK[x_1,x_2]$ with $I = \langle x_2\rangle $ and $J = \langle x_2-x_1x_2^2\rangle$. 
  
  Geometrically, the lemma is giving a situation where the properness of a containment of varieties can be tested by intersecting 
  with a particular hypersurface.
  In general, such intersections might miss additional components of the larger variety, but this does not occur in projective space if we impose the extra condition 
  $I \cap \langle f\rangle = fI$ on the corresponding homogeneous ideals.

     \begin{lemma}  \label{lem:abelian} 
     Suppose $G$ is an additive abelian group with subgroups $A\subseteq B \subseteq G$ and $C\subseteq G$. If $A+C=B+C$ and $A\cap C = B\cap C$ then $A=B$.
     \end{lemma}
     
     \begin{proof}
   Let $b \in B$. 
   As $b \in B+C=A+C$ we have $b = a+c$ for some $a \in A$ and $c \in C$. 
   But then $c = b-a \in B\cap C = A\cap C$ so $c \in A$ whence $b\in A$.
     \end{proof}

\section{Transition systems}\label{ts-sect}

This section contains our main general results (Definition~\ref{ts-def} and Theorem~\ref{new-ts-thm}), concerning an inductive approach to computing initial ideals that we
refer to as a \defn{transition system}.

\subsection{Main definition}

Throughout, we let
$A$ be an affine space over an algebraically closed field $\KK$.
Let $\cN$ be a finite indexing set,
choose variables $x_i$  to identify $\KK[A] = \KK[x_i : i \in \cN]$ with a ring of polynomials, and 
then fix a term order on this ring.

For an ideal $I \subseteq \KK[A]$ and a nonzero element $0\neq f \in \KK[A]$ let 
\be
\label{idqu-eq}
 (I \colon \langle f\rangle) = \tfrac{1}{f} (I \cap \langle f\rangle ) = \{ g \in \KK[A] : fg \in I\}.\ee
This \defn{ideal quotient} is an ideal  satisfying 
\be
\label{idqu-eq2}
I \subseteq (I \colon\langle f\rangle)
\quad\text{with equality if and only if $  I \cap \langle f\rangle =  fI$.}
\ee
Since   $(I \colon\langle f\rangle) = \KK[A]$ if and only if $f \in I$, 
when $I$ is maximal  with $f \notin I$ we have $I=(I \colon\langle f\rangle)$.

\begin{definition}\label{ts-def}
Choose a set $\cT$ of pairs  $(I,J)$ where $I$ and $J$ are proper ideals of $\KK[A]$
such that:
\ben
\item[(T1)] $J$ is a monomial ideal contained in $ \init (I)$ with $J=\init(I)$ if $I$ is maximal.
\item[(T2)] If $(I,J),(I',J') \in \cT$ have $I=I'$ then $J=J'$.
\een
The set $\cT$
is a \defn{transition system} if for each $(I,J) \in \cT$
where $I$ is neither maximal nor equal to $J$, there exists a non-constant 
  $f \in \KK[A]$ with $f\notin I$ and a nonempty finite set $ \Phi \subseteq \cT$ such that:
\ben
\item[(T3)] It holds that $ \textstyle I + \langle f \rangle \subseteq \bigcap_{(P,Q) \in \Phi} P
$ and
$J + \langle \init(f) \rangle = \bigcap_{(P,Q) \in \Phi} Q$.

\item[(T4)] If $ I' := (I \colon \langle f\rangle) $ is not equal to $I$ and $J' := (J\colon \langle \init(f) \rangle)$ then $(I',J') \in \cT$. 

\een
\end{definition}

In this definition, the set $\cT$ may be infinite.
Here is a simple example of this construction.

\begin{example}
Suppose $A=\KK$ so  $\KK[A] = \KK[x]$.
Fix some $c_i \in \KK$ for each $i \in \PP$  and 
define  \[ I_S=  \left\langle{\textstyle \prod_{i \in S}( x-c_i)}\right \rangle
\quand J_S = \langle x^{|S|} \rangle
 \quad\text{for each finite set $\varnothing \subsetneq S \subsetneq \PP$.}\]
We claim that 
$ \cT = \{ (I_S, J_S) : S \in \mathscr{S}\}$ is a transition system
for any subset-closed family $\mathscr{S}$ of nonempty finite subsets of $\PP$.
The ideal $I_S$ is maximal when $|S|=1$, and then $\init(I_{S}) = \langle x \rangle =J_{S}.$
If $|S|>1$ and $m \in S$ is any element, then conditions (T3) and (T4) in Definition~\ref{ts-def}
hold for \[f = x-c_m\quand \Phi  = \left\{ (I_{\{m\}},J_{\{m\}})\right\},\]
since  $(I_S \colon \langle x-c_{m} \rangle)=I_{S\setminus\{m\}}$ and $(J_S \colon \langle \init(x-c_{m} )\rangle) =(\langle x^{|S|} \rangle \colon \langle x\rangle) = \langle x^{|S|-1}\rangle= J_{S\setminus\{m\}}$.
\end{example}

Before discussing more interesting examples, let us present the main application of transition systems: they 
serve as an inductive tool for verifying initial ideal computations.

\begin{theorem}\label{new-ts-thm}
Suppose  $\cT $ is a transition system. Then
$\init( I) = J$ for all $(I,J) \in \cT$.
\end{theorem}

\begin{proof}
Fix a pair $(i,j) \in \cT$.
Our proof is by induction on $I$ ordered by reverse containment. In the base case 
when $I$ is maximal, we have  $J =\init(I)$ by assumption (T1). Also, if $I=J$ then $I$ is a monomial ideal so $J =I = \init(I)$.

 Suppose $I$ is not maximal and $I\neq J$. Assume 
 $f \notin I$ and $\varnothing\subsetneq\Phi\subseteq \cT$ satisfy conditions  (T3) and (T4) in Definition~\ref{ts-def}.
Then 
$
 I\neq I+ \langle f \rangle \subseteq \bigcap_{(P,Q) \in \Phi} P
 $
so  
\be\label{past-two-eq1}  \init (I) + \langle \init(f) \rangle  \subseteq  \init \( I+ \langle f \rangle\) 
\subseteq \textstyle \init \(\bigcap_{(P,Q) \in \Phi} P\)
  \subseteq\textstyle \bigcap_{(P,Q) \in \Phi} \init (P).
\ee
Since $\KK[A]$ is Noetherian and every   $(P,Q) \in \Phi$ has $I\subsetneq I+ \langle f \rangle \subseteq 
 P$, we may assume
by induction
that $\init(P) = Q$ for all $(P,Q) \in \Phi$. As we have $J \subseteq \init(I)$ it follows that
\be\label{past-two-eq2} 
\textstyle \bigcap_{(P,Q) \in \Phi} \init (P) = \textstyle \bigcap_{(P,Q) \in \Phi} Q=J + \langle \init(f) \rangle\subseteq \init(I)+ \langle \init(f)\rangle.
\ee
Thus, the containments in \eqref{past-two-eq1} and \eqref{past-two-eq2} must all be equalities, so
\be\label{must-be-eq}  \init (I) + \langle \init(f)\rangle =  \init \(I + \langle f \rangle\)= J + \langle \init(f) \rangle.\ee

It remains to prove that $J = \init   (I)$. 
First suppose $ (I\colon \langle f\rangle)= I$ so that $ I\cap \langle f\rangle=fI$. Then 
\be\label{uou-eq} \ba \init(f) \init (I)  = \init(f I) 
&= \init (I \cap  \langle f\rangle)  
 = \init (I) \cap  \langle \init(f) \rangle, \ea\ee
 using Proposition~\ref{prop:init-facts}(c)  and \eqref{must-be-eq} for the last equality.
Since $J$ and $\init  (I)$ are monomial ideals, and therefore homogeneous, 
and since $\init(f)$ is a non-constant homogeneous polynomial,
we can use Lemma~\ref{lem:hilbert} to
deduce that $J = \init (I)$. The hypotheses required by the lemma are \eqref{must-be-eq} and \eqref{uou-eq}.

Alternatively, let $I' = (I\colon \langle f \rangle)$ and $J ' = (J\colon \langle \init(f) \rangle)$ and suppose that  $I'\neq I$.
Then   $(I',J') \in \cT$ by Definition~\ref{ts-def} 
but $I \subsetneq I'$, so we may assume by induction that $\init(I') =J'$.
Thus
\[ \ba
\init (I) \cap \langle \init(f) \rangle 
& = \init (I \cap \langle f\rangle) &&\text{by Proposition~\ref{prop:init-facts}(c)  using \eqref{must-be-eq}}
\\& = \init (f I') &&\text{since $ I \cap \langle f \rangle=fI'$ by definition}
\\& = \init(f)  \init (I') 
\\ & = \init(f) J' &&\text{by induction}
\\& = J \cap \langle \init(f)\rangle&&\text{by the definition of $J'$.}
\ea
\]
From this equation and \eqref{must-be-eq}, we can deduce that
 $J = \init   (I)$ by invoking Lemma~\ref{lem:abelian}
 with $A =J$, $B= \init (I)$, and $C=\langle \init(f)\rangle$.
  \end{proof}
 
The proof just given demonstrates that when $\cT$ is a transition system, stronger versions of properties (T3) and (T4) in Definition~\ref{ts-def} hold.
To be precise:

 \begin{corollary}\label{strongerT-cor}
Suppose $(I,J)$ belongs to a transition system $\cT $. Assume  $f \notin I$ is a non-constant polynomial and $ \Phi\subseteq \cT$ is a nonempty finite set  satisfying (T3) and (T4) in Definition~\ref{ts-def}. Also let
\[
 I' := (I \colon \langle f\rangle) \quand J' := (J\colon \langle \init(f) \rangle).\]
Then the following additional properties hold:
\ben
\item[(a)] One has $ I + \langle f \rangle = \bigcap_{(P,Q) \in \Phi} P$ so the containment in (T3) is equality.

\item[(b)] If $I'=I$ then $J'=J$ so the conclusion of (T4) that $(I',J') \in \cT$ holds unconditionally.

\een
 \end{corollary}
 
 \begin{proof}
Since the containments in \eqref{past-two-eq1} are all equalities,
the two sides of the identity in part (a) have the same initial ideal, so are equal by Proposition~\ref{prop:init-facts}(b).
For part (b), 
note that  if $I' = I$,  
then 
\eqref{idqu-eq2} and
  \eqref{uou-eq} 
  show that 
$(\init(I)  \colon \langle \init(f) \rangle) = \init(I) $, while the argument after  \eqref{uou-eq} implies that
  $J=\init(I)$, so we have
$
J' =(J\colon  \langle \init(f) \rangle) =  (\init(I)  \colon \langle \init(f) \rangle)
 =  \init(I) =J
 $ as claimed.
\end{proof}

Another immediate corollary of Theorem~\ref{new-ts-thm} is the following.

\begin{corollary}
For a given set $\sI$ of ideals in $\KK[A]$, there exists at most one transition system $\cT$ with $\sI = \{ I : (I,J) \in \cT\}$.
\end{corollary}

We present one other corollary that relates transition systems to Gr\"obner bases.

\begin{corollary}\label{ts-gr-cor}
Suppose $\cT$ is a transition system and $(I,J) \in \cT$. Then any set $G_I \subseteq I$ with $J \subseteq\langle \init(g) : g \in G_I \rangle$
 is a Gr\"obner basis for $I$. 
\end{corollary}

\begin{proof}
If $(I,J) \in \cT$
then $\init(I)=J$ by Theorem~\ref{new-ts-thm}  so
$G_I$ is a Gr\"obner basis by Corollary~\ref{cor:init-facts}.
\end{proof}

\subsection{Examples}

Below are some other concrete instances of transition systems.
These examples will be considered in greater generality in Sections~\ref{schumat-sect1}, \ref{schumat-sect2}, and \ref{schumat-sect3}.
The cases here are small enough that one can work out all of the relevant details by hand.

In these examples, $A$ will be some affine space of matrices over $\KK$.
Given $w \in A$, we let \[\MX_w^A=\{ M \in A : \rank M_{[i][j]} \leq \rank w_{[i][j]}\text{ for all }i,j\},\]
where $M_{[i][j]}$ means the upper left $i\times j$ submatrix of $M$. 
This \defn{matrix Schubert subvariety} only depends on the \defn{rank table} given by the $\NN$-valued matrix $\left[  \rank w_{[i][j]}\right]$.

\begin{example}\label{ss-ts-ex} Our first example, while particularly simple, helps illustrate the general setup.
Suppose $A$ is the variety of skew-symmetric  $3\times 3 $ matrices  \[A = \left\{ \begin{bsmallmatrix} 0 & -a & -b \\ a &0 & -c \\ b & c& 0 \end{bsmallmatrix} : a,b,c \in \KK\right\}.\]
Then we may identify  $\KK[A]= \KK[u_{21}, u_{31}, u_{32}]$ where $u_{ij}(M) = M_{ij}=-M_{ji}$.
No skew-symmetric $3\times 3$ matrix is invertible, so the only possible rank tables for elements of $A$ are
\[
 \begin{bsmallmatrix} 0 & 0 & 0 \\ 0 &0 & 0 \\ 0 & 0& 0 \end{bsmallmatrix},
\quad
 \begin{bsmallmatrix} 0 & 0 & 0 \\ 0 &0 & 1 \\ 0 & 1& 2 \end{bsmallmatrix},
\quad
\begin{bsmallmatrix} 0 & 0 & 1 \\ 0 &0 & 1 \\ 1 & 1& 2 \end{bsmallmatrix},
\quand
\begin{bsmallmatrix} 0 & 1 & 1 \\ 1 &2 & 2 \\ 1 & 2& 2 \end{bsmallmatrix}.\]
These rank tables arise from various elements, e.g., from the signed permutation matrices
\[
w_1 =  \begin{bsmallmatrix} 0 & 0 & 0 \\ 0 &0 & 0 \\ 0 & 0& 0 \end{bsmallmatrix},
\quad
w_2 =  \begin{bsmallmatrix} 0 & 0 & 0 \\ 0 &0 & -1 \\ 0 & 1& 0 \end{bsmallmatrix},
\quad
w_3 =  \begin{bsmallmatrix} 0 & 0 & -1 \\ 0 &0 & 0 \\ 1 & 0& 0 \end{bsmallmatrix},
\quand
w_4 =  \begin{bsmallmatrix} 0 & -1 & 0 \\ 1 &0 & 0 \\ 0 & 0& 0 \end{bsmallmatrix}.
\]
Define $X_k = \MX^A_{w_k}\subseteq A$ and $I_k = I(X_k)\subseteq \KK[A]$.

We wish to pair each $I_k$ with a monomial ideal $J_k$ to form a transition system.
This is easily done, as $I_k$ is already a monomial ideal.
 To see this, note that we can  informally write 
\[
X_1 = \left\{   \begin{bsmallmatrix} 0 &  \\ 0 & 0 \\ 0 &0 & 0  \end{bsmallmatrix}\right\},
\quad
X_2 = \left\{   \begin{bsmallmatrix}  0 &  \\ 0 & 0 \\ 0 & \ast & 0  \end{bsmallmatrix}\right\},
\quad
X_3 = \left\{   \begin{bsmallmatrix}  0 &  \\ 0 & 0 \\ \ast & \ast & 0 \end{bsmallmatrix}\right\},
\quand
X_4 = \left\{   \begin{bsmallmatrix} 0 &  \\ \ast & 0 \\ \ast &\ast & 0  \end{bsmallmatrix}\right\},
\]
so $I_1 = \langle u_{21}, u_{31}, u_{32}\rangle \supsetneq I _2 = \langle u_{21}, u_{31}\rangle \supsetneq I_3 = \langle u_{21} \rangle \supsetneq I_4 = 0$.
Then $\cT = \{ (I_k,I_k) : k=1,2,3,4\}$ is transition system relative to any term order on $\KK[A]$
since all pairs $(I,J) \in \cT$ have $I=J$.
\end{example}

So far we have only seen transition systems $\cT$ for which we can take the set $\Phi$ in  Definition~\ref{ts-def} to be a singleton for all $(I,J) \in \cT$.
The next example presents a case where this is not possible.

\begin{example}\label{22-ex}
Let $A = \left\{ \begin{bsmallmatrix} a & b \\ c & d \end{bsmallmatrix} : a,b,c,d \in \KK\right\}$ 
so   $\KK[A]= \KK[x_{11}, x_{12}, x_{21}, x_{22}]$ where $x_{ij}(M) = M_{ij}$.
The possible rank tables for elements of $A$ arise from the partial permutation matrices
\[
w_1 =  \begin{bsmallmatrix} 0 & 0 \\ 0 & 0 \end{bsmallmatrix},
\quad
w_2 =  \begin{bsmallmatrix} 0 & 0 \\ 0 & 1 \end{bsmallmatrix},
\quad
w_3 =  \begin{bsmallmatrix} 0 & 0 \\ 1 & 0 \end{bsmallmatrix},
\quad
w_4 =  \begin{bsmallmatrix} 0 & 1 \\ 0 & 0 \end{bsmallmatrix},
\quad
w_5 =  \begin{bsmallmatrix} 0 & 1 \\ 1 & 0 \end{bsmallmatrix},
\quad
w_6 =  \begin{bsmallmatrix} 1 & 0 \\ 0 & 1 \end{bsmallmatrix},
\quad
w_7 =  \begin{bsmallmatrix} 1 & 0 \\ 0 & 0 \end{bsmallmatrix}.
\]
Define $X_k = \MX_{w_k}^A\subseteq A$ and $I_k = I(X_k)\subseteq \KK[A]$.
Choose any term order on $\KK[A]$. Then let 
\[J_k = \init(I_k)\quand \cT=\{(I_k,J_k) : k=1,2,3,4,5,6,7\}.\] We now investigate whether  $\cT$ is a transition system.
As one can informally express  
\[\ba
X_1 &= \left\{   \begin{bsmallmatrix} 0 & 0  \\ 0 & 0 \end{bsmallmatrix}\right\},
\quad
X_2 = \left\{   \begin{bsmallmatrix} 0 & 0 \\ 0 & \ast \end{bsmallmatrix}\right\},
\quad
X_3 = \left\{    \begin{bsmallmatrix} 0 & 0 \\ \ast & \ast \end{bsmallmatrix} \right\},
\quad
X_4 = \left\{   \begin{bsmallmatrix} 0 & \ast \\ 0 & \ast \end{bsmallmatrix}\right\},
\\
X_5 &= \left\{   \begin{bsmallmatrix} 0 & \ast \\ \ast & \ast \end{bsmallmatrix}\right\},
\quad
X_6 = \left\{   \begin{bsmallmatrix} \ast & \ast \\ \ast & \ast \end{bsmallmatrix}\right\},
\quad 
X_7 = \left\{   \det = 0\right\},
\ea
\]
it follows that
\[ \ba 
&I_1 = \langle x_{11}, x_{12}, x_{21}, x_{22}\rangle,  &&I_2 = \langle x_{11}, x_{12}, x_{21}\rangle,\quad 
&& I_3 = \langle x_{11}, x_{12} \rangle, 
&&I_4 = \langle x_{11}, x_{21} \rangle, \\
 &I_5 = \langle x_{11} \rangle && I_6 = 0,
&&  I_7 = \langle x_{11}x_{22}-x_{12}x_{21}\rangle.
\ea
\] Therefore 
$J_k = \init(I_k)= I_k$ for $k\in[6]$ while $ J_7 =\init(I_7)$ is either $ \langle x_{12}x_{21}\rangle$ or $\langle x_{11}x_{22}\rangle$,
depending on the choice of term order.

We only need to verify conditions (T3) and (T4) in Definition~\ref{ts-def}
when $(I,J) = (I_7,J_7)$.
First suppose $J_7 = \langle x_{12}x_{21}\rangle$. 
Then   
  \[ I_7 + \langle x_{11}\rangle =  J_7 + \langle x_{11}\rangle = \langle x_{11},x_{12}x_{21}\rangle = I_3 \cap I_4= J_3 \cap J_4,\]
  so $f=x_{11}$ and $\Phi = \{(I_3,J_3),(I_4,J_4)\}$ satisfy condition (T3) in Definition~\ref{ts-def}. For these choices, condition (T4)  holds vacuously,
  and we conclude that $\cT$ is a transition system.

For term orders with $J_7= \langle x_{11}x_{22}\rangle$, we claim that the set $\cT$ is not a transition system. If it were, then we would have $I_7 + \langle f \rangle = \bigcap_{(P,Q) \in \Phi} P$ for appropriate $\Phi$ and $f$ by Corollary~\ref{strongerT-cor}. Since $f \notin I_7$, we have $(I_7, J_7) \notin \Phi$. The remaining ideals $J_j$ for $j < 7$ are prime, so $J_7 + \langle \init(f) \rangle = \bigcap_{(P,Q) \in \Phi} Q$ is a (possibly redundant) primary decomposition of $J_7 + \langle \init(f) \rangle$. We may assume $(I_1, J_1) \notin \Phi$ since $I_1 \supseteq I_j, J_1 \supseteq J_j$ for all $j$. But then no prime in the decomposition $\bigcap_{(P,Q) \in \Phi} Q$ contains $x_{22}$, whereas $J_7 + \langle \init(f) \rangle = \langle x_{11} x_{22}, \init(f) \rangle$ does have an associated prime containing $x_{22}$.

We can extend $\cT$ to a transition system by adding the ideal pair $(I_8,J_8)$ with  $I_8  = J_8=\langle x_{12}, x_{22}\rangle$
Then, when $(I,J) = (I_7,J_7)$, conditions (T3) and (T4) are satisfied by taking 
$f=x_{12}$ and $\Phi = \{(I_3,J_3),(I_8,J_8)\}$.
The ideal $I_8$ is radical,
but it corresponds to a closed subvariety of $A$ that is not of the form $\MX_w^A$.
 \end{example}

In the next example, a natural way of forming a transition system that includes $I(\MX_w^A)$ for all $w \in A$ does result in the addition of non-radical ideals.

\begin{example}\label{22sym-ex}
Let  $A = \left\{ \begin{bsmallmatrix} a& b \\ b & c \end{bsmallmatrix} : a,b,c \in \KK\right\}$ 
by the variety of symmetric $2\times 2$ matrices over $\KK$
so that   $\KK[A]= \KK[u_{11}, u_{21}, u_{22}]$ with $u_{ij}(M) = M_{ij} = M_{ji}$.  
We will just consider $\KK[A]$ under the reverse lexicographic term order. 
The possible rank tables for elements of $A$   arise from the matrices
\[
w_1 =  \begin{bsmallmatrix} 1 & 0 \\ 0 & 0 \end{bsmallmatrix},
\quad
w_2 =  \begin{bsmallmatrix} 1 & 0 \\ 0 & 1 \end{bsmallmatrix},
\quad
w_3 = \begin{bsmallmatrix} 0 & 1 \\ 1 & 0 \end{bsmallmatrix},
\quad
w_4 =  \begin{bsmallmatrix} 0 & 0 \\ 0 & 1 \end{bsmallmatrix},
\quad
w_5 =   \begin{bsmallmatrix} 0 & 0 \\ 0 & 0 \end{bsmallmatrix}.
\]
Let $X_k = \MX_{w_k}^A\subseteq A$ and $I_k = I(X_k)\subseteq \KK[A]$. We can informally write the subvarieties $X_k$ as 
\[
X_1 = \left\{   \det=0\right\},
\quad
X_2 = \left\{   \begin{bsmallmatrix} \ast &  \\ \ast & \ast \end{bsmallmatrix}\right\},
\quad
X_3 = \left\{   \begin{bsmallmatrix} 0 &  \\ \ast & \ast \end{bsmallmatrix}\right\},
\quad
X_4 = \left\{   \begin{bsmallmatrix} 0 &  \\ 0 & \ast \end{bsmallmatrix}\right\},
\quad
X_5 = \left\{   \begin{bsmallmatrix} 0  &  \\ 0 & 0 \end{bsmallmatrix}  \right\},
\]
and so $I_1 = \langle  u_{21}^2 - u_{11} u_{22}\rangle$ and   $I_2 =0\subsetneq I_3 =\langle u_{11}\rangle \subsetneq I_4 = \langle u_{11}, u_{21}\rangle\subsetneq I_5 = \langle u_{11}, u_{21}, u_{22} \rangle$.
Define 
\[ J_1 =\init(I_1)= \langle u_{21}^2\rangle \quand J_k =\init(I_k)= I_k\text{ for $k\in\{2,3,4,5\}$}.\]
The set  $\{ (I_k,J_k) : k =1,2,3,4,5\}$ fails to be a transition system, though only by a small margin.

We only need to verify conditions (T3) and (T4) in Definition~\ref{ts-def}
when $(I,J) = (I_1,J_1)$ since $I_k=J_k$ when $k>1$.
The problem is that 
$\textstyle J_1 + \langle \init(f)\rangle $ is never equal to  $\bigcap_{k \in S} J_k$ 
for any choice of set  $\varnothing\subsetneq S\subseteq \{1,2,3,4,5\}$ when $f \in \KK[A]$ is non-constant with $f\notin I_1$. Indeed, if the set above formed a transition system we would have $1 \notin S$ as in Example~\ref{22-ex}, and $2 \notin S$ since $J_1 + \langle \init(f) \rangle \neq 0$. Then $J_1 + \langle \init(f) \rangle = \bigcap_{k \in S} J_k$ contains $u_{11}$, so $\init(f) = u_{11}$ and $J_1 + \langle \init(f) \rangle = \langle  u_{11}, u_{21}^2 \rangle$, which is evidently not equal to any intersection of $J_3, J_4, J_5$.

To get around this, define
\[
I_6 =J_6= \langle u_{11},u_{21}^2\rangle =  I_1 + \langle u_{11} \rangle = J_1 + \langle u_{11}\rangle.\]
Now if $(I,J) = (I_1,J_1)$ then (T3) holds for $f=u_{11} \notin I_1$ and $\Phi = \{(I_6,J_6)\}$, and $(I\colon \langle f\rangle)=I$ so (T4) holds vacuously.
Thus the set
$\cT=\{ (I_k,J_k) : k \in [6]\}$ is a transition system.
Notice, however, that the extra ideal $I_6$ is not radical, so corresponds to a non-reduced subscheme of $A$.
\end{example}

\subsection{Orbit closures}\label{orbit-sect}

We are most interested in transition systems $\cT$ made up of pairs $(I,J)$
where $I$ ranges over the \defn{orbit closures}  of certain maximal ideals. This section studies the relevant orbit closure operation. It is convenient to present these results in an entirely self-contained way, but generalizations of this material in terms of equivariant sheaves may be well-known; see, e.g., \cite{ChGi,MillerSturmfels}. If the reader is willing to believe that obvious properties of orbit closures of subvarieties extend naturally from radical ideals to all ideals, the details of this section can be skipped.

Continue to let $A$ be an affine space over an algebraically closed field $\KK$.
Suppose $G$ is a connected linear algebraic group over $\KK$ that acts algebraically on $A$.
We identify the coordinate ring $\KK[G\times A] $ with  $\KK[G]\otimes \KK[A]$ where $\otimes$ denotes the tensor product over $\KK$.
The action of $G$ on $A$ determines a map
\be\label{pullback-eq}
\barr{rrcl} \alpha : &G\times A &  \to&  A \\ & (g,a) & \mapsto & g\cdot a
\earr
\quad\text{whose pullback is }\quad
\barr{rrcl}
 \alpha^* : & \KK[A] &\to& \KK[G\times A]
 \\
 & f & \mapsto & f\circ \alpha.
 \earr
 \ee
Since the action of $G$ on $A$ is algebraic, $\alpha$ is a morphism of algebraic varieties and
 $\alpha^*$ is a ring homomorphism.
 In addition, $\alpha^\ast$ is injective
as $\alpha^*(f)$ sends $(1,a) \mapsto f(a)$ for all $a \in A$.

When $I \subseteq \KK[A]$ is an ideal, we denote the ideal of  $\KK[G\times A]$ generated by the set $\{ 1\otimes f : f \in I\}$ as
$\KK[G] \otimes I$.
Since $\KK[G]$ is a free and flat $\KK$-module, the following holds:

\begin{lemma} \label{flat-lem}
If $I$ is any ideal in $\KK[A]$ then $\KK[G\times A] / (\KK[G]\otimes I) \cong \KK[G] \otimes (\KK[A]/I) $.
\end{lemma}

Now, the  \defn{orbit closure} of an ideal $I\subseteq\KK[A]$, relative to the action of $G$,  is defined to be  \be\cl_G(I) = (\alpha^*)^{-1}(\KK[G] \otimes I) \subseteq \KK[A].\ee
Because $\alpha^*$ is a ring homomorphism,
this preimage is an ideal in $\KK[A]$.

\begin{example}
If $I=0$ then $\cl_G(0) = 0$ and if $I = \KK[A]$ then $\cl_G(I)= \KK[A]$.
\end{example}

For any ideal $I \subseteq \KK[A]$  define 
$ V(I) := \{ a \in A : f(a)=0\text{ for all }f \in I\}$
and for any subset $V\subseteq A$  define $ I(V) := \{ f \in \KK[A] : f(v)=0\text{ for all } v \in V\}.$
Because $\KK$ is algebraically closed,
Hilbert's Nullstellensatz asserts that
 if $I$ is any ideal then $V(I) $
is a closed affine subvariety and $I(V(I)) = \sqrt{I}$ is equal to the \defn{radical} of $I$.

The ideal $\cl_G(I)$ is related to the closure of the union of the $G$-orbits through all points in $V(I)$.
 
\begin{proposition}\label{rad-prop}
If $I$ is a radical ideal in $\KK[A]$ then so is $ \cl_G(I)$,
and in this case $V(\cl_G(I))$ is the closure of the set $\{ g \cdot a : (g,a) \in G \times V(I)\}$ in the Zariski topology of $ A$.
\end{proposition}

\begin{proof}
An ideal $I$ in a commutative ring $R$ is radical if and only if the quotient ring $R/I$ 
is \defn{reduced} in the sense of having no nonzero nilpotent elements.
Since a linear algebraic group is an affine variety, the coordinate ring $\KK[G]$ is an integral domain and therefore reduced.
On the other hand, 
by Lemma~\ref{flat-lem} 
we know that 
$\KK[G\times A] / (\KK[G]\otimes I)\cong  \KK[G] \otimes (\KK[A]/I) $.  
Tensor products of finitely generated reduced algebras over perfect fields are always reduced 
\cite[Thm.~3, Chapter V, \S15]{Bourbaki}, and our field $\KK$ is perfect since it is algebraically closed.
Thus if $I$ is radical then both $\KK[A]/I$ and $\KK[G]$ are reduced,
and hence their tensor product
 $ \KK[G] \otimes (\KK[A]/I)$ is also reduced.
In this case $\KK[G]\otimes I$ is radical, as is its preimage $ \cl_G(I)$ under the ring homomorphism $\alpha^*$.
 
Now let $\cO = \{ g \cdot a : (g,a) \in G \times V(I)\}\subseteq A$. 
If $f \in  \cl_G(I)$ then $\alpha^*(f) \in \KK[G] \otimes I $, so for all $g \in G$ and $a \in V(I)$ we have 
$f(g\cdot a) =  \alpha^*(f)(g,a) = 0$. This shows that $ \cl_G(I) \subseteq I(\cO)$ so $\cO \subseteq V(\cl_G(I))$.
Since $V( \cl_G(I))$ is a closed affine subvariety, it also contains the closure of $\cO$.

Conversely, if $J\subseteq \KK[A]$ is any radical ideal with $\cO \subseteq V(J)$,
then each $f \in J$
has
\[
\alpha^*(f) \in I(G\times V(I)) = I(G) \otimes \KK[A] + \KK[G] \otimes I(V(I)) 
=  \KK[G]\otimes I.
\]
Thus $J \subseteq \cl_G(I)$ and $V(\cl_G(I)) \subseteq V(J)$,
so the closure of $ \cO$, which is the intersection of all such sets $V(J)$, also contains $V(\cl_G(I))$.
\end{proof}

%
%
%

We note a few general properties of the orbit closure operation $\cl_G$.

\begin{lemma} \label{lem:subset-closure}
    If $I \subseteq \KK[A]$ is any ideal then $ \cl_G(I)\subseteq I$.
\end{lemma}

\begin{proof}
Let $\iota : A \to G \times A$   be the map   $a \mapsto (1,a)$.
Then  
$\iota^*(f_1 \otimes f_2)=  (f_1\otimes f_2)\circ \iota= f_1(1)f_2$ when $f_1 \in \KK[G]$ and $f_2 \in \KK[A],$
so  $\iota^*$ is a surjection $\KK[G]\otimes I \to I$. Hence, if
 $f \in \cl_G(I)$ then  
 $\iota^* \circ \alpha^*(f) \in \iota^*(\KK[G] \otimes I) =  I.$ But $\iota^* \circ \alpha^* = (\alpha \circ \iota)^* = \id_{\KK[A]}$,
 so if $f \in \cl_G(I)$ then $f \in I$. 
 \end{proof}

We define an ideal $I\subseteq \KK[A]$ to be \defn{$G$-stable} if $\cl_G(I) = I$. 
The group $G$
acts on $\KK[A]$ by the formula $g\cdot f : a \mapsto f(g^{-1} \cdot a)$
for $g \in G$, $f \in \KK[A]$, and $a \in A$.
Our notion of $G$-stability can alternatively be characterized in terms of this action.
 
\begin{proposition}\label{stab-prop} Consider an ideal $I \subseteq \KK[A]$. Then the following properties hold:
\ben
\item[(a)] The orbit closure ideal $\cl_G(I)$ is $G$-stable in the sense that $\cl_G(\cl_G (I)) = \cl_G(I)$.
\item[(b)] The ideal $I$ is $G$-stable if and only if $g\cdot I = I$ for all $g \in G$.
\een
\end{proposition}

\begin{proof}
Fix $g \in G$ and suppose $f \in \cl_G(I)$. Then we can write $\alpha^*(f) = \sum_i p_{i} \otimes q_{i}$ as a finite sum
with each $p_{i} \in \KK[G]$ and $q_{i} \in I$. The group $G$ also acts on $\KK[G]$, and since 
\[ \alpha^*(g\cdot f) (x,y) = f(g^{-1}x\cdot y) = \alpha^*(f)(g^{-1}x,y)\quad\text{for }x \in G\text{ and }y \in A,
\]
we have 
$ \textstyle \alpha^*(g\cdot f) = \sum_i (g\cdot p_{i}) \otimes q_{i} \in \KK[G]\otimes I.$
We conclude that $g\cdot f\in\cl_G(I)$, so $g\cdot \cl_G(I) \subseteq \cl_G(I)$.
As $G$ is a group and $g \in G$ is arbitrary, this implies that $g\cdot  \cl_G(I) =  \cl_G(I)$ for all $g \in G$. 

 Thus, part (a) will follow if we can prove the ``if'' direction of part (b). 
Suppose $I$ is an ideal with $g\cdot I = I $ for all $g \in G$. 
In view of Lemma~\ref{lem:subset-closure} it suffices to show that $I \subseteq \cl_G(I)$.
Let $f \in I$.
We must check that $\alpha^*(f) \in \KK[G]\otimes I$.

We can write $\alpha^*(f) = \sum_{i=1}^m p_i \otimes q_i$ where $p_1,p_2,\dots,p_m \in \KK[G]$
are linearly independent over $\KK$ and each $q_i \in \KK[A]$.
For $g\in G$ let $\iota_g : A \to G\times A$ denote the map $a\mapsto (g,a)$. Then
 \[\textstyle \iota_g^*\circ \alpha^*(f) = \alpha^*(f) \circ \iota_g= \sum_{i=1}^m p_i(g) q_i \in \KK[A].\]
At the same time, for each $g \in G$ we have 
$\iota_g^*\circ \alpha^*(f) = f\circ \alpha\circ \iota_g = g^{-1}\cdot f$ which is in $I$ by hypothesis.

Finally, observe that as $p_1,p_2,\dots,p_m$ are linearly independent as functions on $G$, we may choose group elements
$g_1,g_2,\dots,g_m \in G$ such that the matrix $\left[ p_i(g_j) \right]_{1\leq i,j \leq m}$ is invertible.
Then the fact that we have $\sum_{i=1}^m p_i(g_j) q_i = \iota_g^*\circ \alpha^*(f) \in I$ for all $j$ implies that the functions
$q_1,q_2,\dots,q_m$ are all in $I$, and so  $\alpha^*(f)=\sum_{i=1}^m p_i \otimes q_i \in \KK[G]\otimes I$ as needed.

This proves part (a) and the ``if'' direction of part (b).
The ``only if'' direction holds as if $I$ is $G$-stable then $I = \cl_G(I)$ and we have already checked that
$g\cdot\cl_G(I) = \cl_G(I)$ for all $ g\in G$.
\end{proof}

The previous result shows that
 $\cl_G(I)$ is $G$-stable and contained in $I$.
In fact, $\cl_G(I)$ is the largest  ideal with this property.

\begin{proposition} If $I \subseteq \KK[A]$ is any ideal then $\cl_G(I)$ is the sum of all $G$-stable ideals in $ I$.
\end{proposition}

\begin{proof}
    Let $\Sigma $ be the sum of all $G$-stable ideals contained in $I$. Since $\cl_G (I)$ is $G$-stable and contained in $I$ by Lemmas~\ref{lem:subset-closure} and Proposition~\ref{stab-prop}, we have $\cl_G(I) \subseteq \Sigma $. Conversely, if $J \subseteq I$ then it is clear that  $\cl_G(J) \subseteq \cl_G (I)$, so if $J$ is also $G$-stable then  $J = \cl_G (J) \subseteq \cl_G (I)$. Hence $\Sigma  \subseteq \cl_G (I)$.
\end{proof}

 An ideal $I$ in  $R$ is \defn{primary} if  
all zero divisors in  $R/I$ are nilpotent. 
To conclude this section, we show that 
the closure operation $\cl_G$   preserves this property.

%
%

\begin{proposition} \label{prop:primary-tensor}
If $I$ is a primary ideal in $ \KK[A]$ then $\KK[G]\otimes I$ and $\cl_G(I)$ are both primary.
\end{proposition}

\begin{proof}
Assume $I$ is primary and 
let $f_1,f_2 \in \KK[G]\otimes \KK[ A]/I \cong \KK[G\times A]/(\KK[G]\otimes I)$ be nonzero elements with $f_1f_2= 0$. Write $f_2 = \sum_{i=1}^m p_{i} \otimes q_{i}$ with $p_{i} \in \KK[G]$ and the $q_{i} \in \KK[A]/I$ linearly independent. For each $g \in G$ again let $\iota_g : A \to G \times A$ be the map $a \mapsto (g,a)$, so that $\iota^\ast_g(f_2)=f_2\circ \iota_g = \sum_{i=1}^m p_{i}(g)q_{i} \in \KK[A]/I$.

Let $U$ be the set of $g \in G$ with $\iota_g^\ast(f_2)\neq 0$. Then $U = \bigcup_{i=1}^m \{ g \in G: p_{i}(g)\neq 0\}$ is a nonempty open subset, so as $G$ is connected and irreducible, $U$ is dense in $G$.
Consider some $g \in U$. Then we have $\iota_g^\ast(f_1)\iota_g^\ast(f_2) = \iota_g^\ast(f_1f_2) = 0$ and $\iota_g^\ast(f_2)\neq 0$,
so $\iota_g^\ast(f_1)$ is nilpotent  as $I$ is primary. Therefore
\[ \iota_g^\ast(f_1)^N = \iota_g^\ast(f_1^N)=0\]
for some integer $N>0$, which can be chosen 
independently of $g$ as $\KK[A]/I$ is Noetherian.

Now write $f_1^N =  \sum_{i=1}^{m'} p'_{i} \otimes q'_{i}$ with $p'_{i} \in \KK[G]$ and the $q'_{i} \in \KK[A]/I$ linearly independent.
Then we must have $p'_i(g) = 0$ for all $i$ since $0=\iota_g^\ast(f_1^N)=  \sum_{i=1}^{m'} p'_{i}(g) q'_{i}$.
As this holds for all $g$ in the dense set $U$, it follows that $p'_i = 0$ for all $i$ so $f_1^N=0$. 
In other words, the zero divisor $f_1 \in \KK[G\times A]/(\KK[G]\otimes I)$ is nilpotent.

We conclude that $\KK[G]\otimes I$  is a primary ideal in $\KK[G\times A]$. Since $\cl_G(I)$ is the preimage of this ideal under the morphism $\alpha^\ast$, it is also primary.
\end{proof}

\subsection{Reductions}

Transition systems that arise from orbit closures of maximal ideals 
 are simpler than the general case in two respects.
First, we can usually choose $f \notin I$ in Definition~\ref{ts-def} to be one of the variables $x_i$ generating $\KK[A]$.
Second, condition (T4) in Definition~\ref{ts-def} will often hold by default.

We prove two technical lemmas in this section in order to efficiently make these reductions later. 
Continue to let $G$ be a linear algebraic group that acts algebraically on an affine space $A$,
with $G$ and $A$ both defined over any algebraically closed field $\KK$.
Recall that we identify $\KK[A]$ with the polynomial ring $\KK[x_i : i \in \cN]$ where $\cN$
is some finite indexing set.

\begin{lemma}\label{tlem1}
Suppose there exists a partial order $\preceq$ on $\cN$ 
such that if $j \in \cN$ then
\[\textstyle \alpha^\ast(x_j) = \sum_{i\preceq j} f_{ij} \otimes x_i \in\KK[G\times A] \quad\text{for elements $f_{ij} \in \KK[G]$
with $f_{jj}$ invertible}.\]
Let  $I$ be an ideal of $ \KK[A]$ and suppose $j \in \cN$ has the property that 
$x_i \in I$ for all $i \prec j$. Then:
\ben
\item[(a)] We have $x_j \notin \cl_G(I)$ if and only if $x_j \notin I$.
\item[(b)] It holds that 
$
(\cl_{G}(I) \colon \langle x_j\rangle)
=
\cl_{G}( (I \colon \langle x_j\rangle) )
.$
\een
\end{lemma}

\begin{proof}
If $x_j \notin I$ then $x_j\notin \cl_G(I)$ since  $\cl_G(I)\subseteq  I$.
Conversely, if $x_j \in I$ then $x_i \in I$ for all $i\preceq j$ so by hypothesis $\alpha^\ast(x_j) \in \KK[G]\otimes I$ and $x_j \in \cl_G(I)$. This proves part (a).

To prove part (b), let $I'=(I \colon \langle x_j\rangle)$. 
Note that if $x_j \in I$ then we have $I' = \KK[A]=\cl_G(I') $ and also $(\cl_G(I) \colon \langle x_j \rangle) = \KK[A]$ since $x_j \in \cl_G(I)$ by part (a). We may therefore assume that $x_j \notin I$, although this condition does not play a role in the following argument.
To show that $
(\cl_{G}(I) \colon \langle x_j\rangle)
=
\cl_{G}(I') 
$ we check that each ideal contains the other.
First let $F \in \cl_G(I')$.
Then 
\[\textstyle\alpha^\ast (x_j F) = \alpha^\ast (x_j) \alpha^\ast(F) = \sum_{i\prec j} \underbrace{(f_{ij} \otimes x_i)}_{ \in \KK[G]\otimes I}  \alpha^\ast(F)+ \underbrace{(f_{jj}\otimes x_j)
 \alpha^\ast(F)}_{\in \KK[G]\otimes x_j I'}.\]
 As we have $x_j I' = I \cap \langle x_j\rangle \subseteq I$, it follows that $\alpha^\ast (x_j F)\in \KK[G]\otimes I$,
 so we deduce that $x_j F \in \cl_G(I)$ and therefore $F \in (\cl_G(I) \colon \langle x_j\rangle)$. This shows that $\cl_G(I') \subseteq (\cl_G(I) \colon \langle x_j\rangle)$.
 
 For the reverse containment, consider a generic element of $\cl_G(I) \cap \langle x_j\rangle$,
 which must have the form $x_j F$ for some $F \in \KK[A]$  with $\alpha^\ast(x_jF) \in  \KK[G]\otimes I$. 
 Since we can write 
 \[ \textstyle\alpha^\ast(x_jF)=\alpha^\ast(x_j)\alpha^\ast(F) =\sum_{i \prec j} \underbrace{(f_{ij} \otimes x_i)}_{\in\KK[G]\otimes I} \alpha^\ast(F) + (f_{jj} \otimes x_j)\alpha^\ast(F),\]
it follows that  $(f_{jj} \otimes x_j)\alpha^\ast(F) \in \KK[G]\otimes I$.
As $f_{jj}$ is invertible, we  have $(1\otimes x_j) \alpha^\ast(F) \in \KK[G]\otimes I$. 
This element is also in $ \KK[G]\otimes \langle x_j\rangle$,
and so 
\be\label{flat-eq1} (1\otimes x_j) \alpha^\ast(F)\in  (\KK[G]\otimes I) \cap (\KK[G] \otimes \langle x_j\rangle).
\ee
Recall that we view both $\KK[G]\otimes I$ and $\KK[G] \otimes \langle x_j\rangle$ as submodules of $\KK[G\times A]$, so the intersection on the right is well-defined.
Under this convention, it holds that  
\be\label{flat-eq2}  (\KK[G]\otimes I) \cap (\KK[G] \otimes \langle x_j\rangle)=\KK[G] \otimes (I\cap \langle x_j\rangle) = \KK[G] \otimes x_j I'\ee
since $\KK[G]$ is a free $\KK$-module (as it is well-known for modules over a commutative ring that  tensoring with a flat module commutes with finite intersections). 
From \eqref{flat-eq1} and \eqref{flat-eq2} we get
\[(1\otimes x_j) \alpha^\ast(F) \in  \KK[G] \otimes x_j I' = (1 \otimes x_j)(\KK[G] \otimes I').\]
Since $\KK[G \times A]$ is an integral domain
(as $G$ is a connected linear algebraic group, hence irreducible, and $A$ is an affine space)
 this implies $\alpha^\ast(F) \in \KK[G] \otimes I'$. We conclude that $F \in \cl_G(I')$ and $x_j F \in x_j \cl_G(I')$.
 This means that $\cl_G(I) \cap \langle x_j\rangle \subseteq x_j \cl_G(I')$ so $(\cl_G(I) \colon \langle x_j\rangle) \subseteq \cl_G(I')$.
  \end{proof}
  
  Our second lemma is a straightforward property of maximal ideals in polynomial rings.
  
  \begin{lemma}\label{tlem2}
  Suppose 
  $I = \langle x_i - c_i : i \in \cN \rangle \subset\KK[A]$ for certain constants $c_i \in \KK$ with $ i \in \cN$,
so that $I$   is the (maximal) vanishing ideal of a point in $A$.
Then for any given $j \in \cN$ the following properties are equivalent:
(a) $x_j \notin I$, (b) $(I \colon \langle x_j \rangle) =  I$, and  (c) $c_j\neq 0$.
\end{lemma}

\begin{proof}
Properties (a) and (b) are equivalent since $I$ is a maximal ideal, in view of the observations before Definition~\ref{ts-def}.
If $x_j \notin I$ then clearly $c_j \neq 0$.
Conversely, if $c_j \neq 0$ then we cannot have $x_j \in I$ as then   $1 = \tfrac{1}{c_j}( x_j-(x_j-c_j)) \in I$,
contradicting the assumption that $I$ is a proper ideal.
  \end{proof}

\section{Matrix Schubert varieties}\label{schumat-sect1}

This is the first of two semi-expository sections outlining applications of Theorem~\ref{new-ts-thm}.
Fix positive integers $m,n\in\PP$ and let
$ \Mat_{m\times n}$ be the variety of $m\times n$ matrices over an algebraically closed field $\KK$.
We will discuss the following generalizations of the varieties in Example~\ref{22-ex}:

\begin{definition}
For each $m\times n$ matrix $w$ over $\KK$, the corresponding 
  \defn{matrix Schubert cell}
  and  \defn{matrix Schubert variety}
are the respective subsets of $\Mat_{m\times n}$ defined by the rank conditions
   \be \begin{aligned}
    \openMX_w &= \left\{ M \in \Mat_{m\times n} : \text{$\rank M_{[i][j]} = \rank w_{[i][j]}$ for $(i,j) \in [m]\times [n]$}\right\},
    \\
    \MX_w &= \left\{ M \in \Mat_{m\times n} : \text{$\rank M_{[i][j]} \leq \rank w_{[i][j]}$ for $(i,j) \in[m]\times [n]$}\right\},
\end{aligned}
\ee
where $M_{[i][j]}$ stands for the upper left $i\times j$ submatrix of $M$. 
\end{definition}

To study matrix Schubert varieties, one is naturally led to investigate the vanishing ideals $I(\MX_w)$.
From this perspective, it is of interest to compute the initial ideal of $I(\MX_w)$ and then to construct a corresponding Gr\"obner basis. We explain below how to do such calculations using transition systems. This will recover several results from \cite{KnutsonMiller}.

The transition system approach is best suited to term orders on $\KK[\Mat_{m\times n}]$ with a certain \defn{antidiagonal} property defined in \cite[\S1.5]{KnutsonMiller}. For simplicity, we will always work with the following instance of such an order.
First, identify the coordinate ring of $\Mat_{m\times n}$ with 
\[\KK[\Mat_{m\times n}] = \KK[x_{ij} : 1\leq i \leq m, 1\leq j \leq n]\]
by setting $x_{ij}$ to be the linear function on matrices given by $x_{ij}(M) = M_{ij} \in \KK$.
Then order the monomials in this ring using the reverse lexicographic term order explained in Remark~\ref{rlto-rem}.

\begin{remark}
Prior to \cite{KnutsonMiller}, the problem of finding Gr\"obner bases  
for various special cases of the ideals $I(\MX_w)$ had already appeared in a number of places; see \cite[\S2.4]{KnutsonMiller} for a historical overview.
Since \cite{KnutsonMiller}, there has only been partial progress on understanding the Gr\"obner geometry of 
matrix Schubert varieties for arbitrary term orders; see, for example,  \cite{GaoYong,HPW,KleWei,KMY}.
\end{remark}

\subsection{A reverse lexicographic transition system}\label{msv-sect1}

The variety $\MX_w$ is   irreducible and equal to the Zariski closure of  $\openMX_w$ \cite[Ch.~15]{MillerSturmfels}.
 Let $\B_m$ and $\B_n$ be the Borel subgroups of invertible lower-triangular matrices in $\Mat_{m\times m}$ and $\Mat_{n\times n}$. 
Then $ \B_m\times \B_n$ acts algebraically on $ \Mat_{m\times n}$ by 
 $ (g,h)\cdot M = g M h^{\top}$, and  $\openMX_w$ is the orbit of $w$ under this action.
It follows that in the notation of Section~\ref{orbit-sect}, we can express
\be\label{msv-m-eq} I(\MX_w) = \cl_{\B_m\times \B_n}(\cM_w)
\quad\text{for the maximal ideal }
 \cM_w := \langle x_{ij} - w_{ij} : (i,j) \in [m]\times[n]\rangle.\ee

As a matrix has rank $k$ if and only if all of its $(k+1)\times(k+1)$ minors are zero, 
$I(\MX_w)$  contains certain minors of the matrix of variables $\cX = [x_{ij}]_{1\leq i \leq m, 1\leq j \leq n}$. Specifically, if we define 
\be\label{msv-minors-eq}
I_w := \left\langle \det(\cX_{RC}) : i \in [m], \ j \in [n],\ R \subseteq [i],\ C\subseteq[j],\ |R|=|C|=1+\rank w_{[i][j]} \right\rangle
\ee
then it is clear \emph{a priori} that $I_w \subseteq I(\MX_w)$.
We will see later (in Corollary~\ref{msv-cor}) that this containment is actually equality.

For a $k\times k$ matrix $M$ define  
$\adiag(M) := M_{1,k}M_{2,k-1}M_{3,k-2}\cdots M_{k,1}$ to be the product of the entries on the antidiagonal.
Then we also consider the monomial ideal
\be\label{jdef-eq}
J_w := \left\langle \adiag(\cX_{RC}) : i \in [m], \ j \in [n],\ R \subseteq [i],\ C\subseteq[j],\ |R|=|C|=1+\rank w_{[i][j]} \right\rangle.
\ee
Observe that $ \adiag(\cX_{RC})$ is the leading term of $\det(\cX_{RC}) $ under the reverse lexicographic term order.
Therefore, relative to this term order, we automatically have $J_w \subseteq \init(I_w) \subseteq \init(I(\MX_w))$.

\begin{theorem} \label{msv-thm}
The set $\cT=\{(I(\MX_w),J_w) : w \in \Mat_{m\times n}\}$
is a finite transition system.
\end{theorem}
 
We outline the proof of this theorem below.
Before commencing this, we first explain some of the additional data that goes into this statement, in particular:
(a) a finite indexing set for $\cT$, (b) how to identify the ideals $I(\MX_w)$ that are maximal, (c) for non-maximal $I(\MX_w)$ 
how to construct the non-constant polynomial $f  \notin I(\MX_w)$ and subset $\Phi\subseteq \cT$ 
required in Definition~\ref{ts-def}.
 
 \subsection{Finite indexing set}\label{msv-pp-sect}

Observe that   $\MX_w = \MX_{gwh^\top}$, $I_w = I_{gwh^\top}$, and $J_{w} = J_{gwh^\top}$ for all $(g,h) \in \B_m\times \B_n$ and $w \in \Mat_{m\times n}$.
Meanwhile, the $\B_m\times \B_n$-orbit of any element of $\Mat_{m\times n}$ contains a unique $m\times n$ \defn{partial permutation matrix} \cite[Ch.~15]{MillerSturmfels}, so the size of $\cT$ in Theorem~\ref{msv-thm} is   at most the number of such matrices.
The following notation is convenient to exploit these observations.

Let $S_{\infty}$ be the group of permutations of $\NN$ that fix all but finitely many points.
We define  $S_n$ to be the subgroup $S_n = \{ w \in S_\infty : w(i)=i\text{ for all } i>n\}$.
The \defn{descent sets} of
$w \in S_\infty$ are 
$\DesR(w) := \{ i \in \PP: w(i) > w(i+1)\}$ and $ \DesL(w) :=\DesR(w^{-1}).$
The set of permutations  \be  S^{(m,n)}_{\infty} := \{w \in S_{\infty} : \DesR(w) \subseteq [m]\text{ and }\DesL(w) \subseteq [n]\}\ee
 is in bijection with the finite set of $m\times n$ partial permutation matrices,
via the map that identifies $w \in S^{(m,n)}_{\infty}$ with the $m\times n$ matrix having $1$ in position $(i,w(i))$
for each $i \in [m]$ and $0$ elsewhere.
We freely make this identification in order to define $\MX_w$, $I_w$, and $J_w$ for $w \in  S^{(m,n)}_{\infty}$. Then  
\be\cT =\left\{(I(\MX_w),J_w) : w \in \Mat_{m\times n}\right\}=\left\{(I(\MX_w),J_w) : w \in  S^{(m,n)}_{\infty} \right\}\ee
and the set on the right is uniquely indexed,
since if $v,w \in   S^{(m,n)}_{\infty}$ then $I(\MX_v) \subseteq I(\MX_w)$
 if and only if $v\leq w$ in the \defn{Bruhat order} on $S_\infty$ \cite[Thm.~2.1.5]{CCG}.

\begin{example}\label{0mn-eq}
The ideal $I(\MX_w)$ is maximal if and only if $\MX_w$ is a point, which occurs only when $\MX_w = \{0\}$
as $\KK^\times = \KK\setminus\{0\}$ acts on $\MX_w$ by multiplication.
In this case the index   $w \in S_\infty^{(m,n)}$ is 
$0_{m\times n} := (n+1)(n+2)\cdots(m+n)12\cdots n \in S_\infty$
and the corresponding ideals are irrelevant in the sense that 
$ I(\MX_{0_{m\times n}}) = I_{0_{m\times n} } = J_{0_{m\times n} } = \init(I(\MX_{0_{m\times n} })) = \langle x_{ij} : (i,j) \in [m]\times[n]\rangle.$
\end{example}

\subsection{Proof of Theorem~\ref{msv-thm}}\label{msv-non-sect}

The \defn{dominant component} of any bijection $w \in \PP \to \PP$ is the set of pairs 
\be\label{dom-eq} \dom(w) := \left\{ (i,j) \in\PP\times \PP : \rank w_{[i][j]}=0\right\},\ee
where we compute $ \rank w_{[i][j]}$ by
identifying $w$ with the (infinite) permutation matrix having $1$ in each position $(i,w(i))$.
The set $\dom(w)$ always coincides with the \defn{Young diagram} \be\D_\lambda := \{ (i,j) \in \PP\times \PP : 1\leq j\leq \lambda_i\}\ee of some integer partition $\lambda = (\lambda_1\geq \lambda_2\geq \dots \geq 0)$.
An \defn{outer corner} of $\dom(w)$ is a pair $(i,j) \in \PP\times \PP$ such that $\dom(w)\subsetneq \dom(w) \sqcup\{(i,j)\} = \D_\mu$  
for some integer partition $\mu$.
 
\begin{example}
Consider $w = 43152 \in S_5$ written in one-line notation. Then $\DesR(w) =\{1,2,4\}$ and $\DesL(w) = \{2,3\}$ so $w \in S_\infty^{(m,n)}$ for any $m\geq 4$ and $n\geq 3$. The 
dominant component of $w$ is 
\[ \dom(w) =  \left\{\begin{smallmatrix}\square&\square&\square&1 & \cdot\\ \square&\square&1 & \cdot &\cdot \\ 1 & \cdot & \cdot & \cdot & \cdot \\ \cdot & \cdot & \cdot & \cdot & 1 \\ \cdot & 1 & \cdot & \cdot & \cdot \end{smallmatrix}\right\} = \{(1,1),(1,2),(1,3),(2,1),(2,2)\} = \D_\lambda\] for $\lambda=(3,2)$. The outer corners of $\dom(w)$ are $(1,4)$, $(2,3)$, and $(3,1)$.
\end{example}

\begin{proposition}\label{exerprop}
Suppose $w \in S_\infty^{(m,n)}$ and $(i,j) \in [m]\times [n]$.
Then $(i,j) \in \dom(w)$ if and only if $x_{ij} \in I(\MX_w)$ if and only if $x_{ij} \in J_w$.
\end{proposition}

\begin{proof}
Assume $(i,j) \in \dom(w)$. Then $\rank M_{[i][j]}=\rank w_{[i][j]}=0$, so $M_{ij}=0$ for all $M \in \MX_w$,
and therefore $x_{ij} \in I(\MX_w)$. We also have $x_{ij} \in J_w$ as $x_{ij} = \adiag(\cX_{RC})$ 
 for $R = \{i\}$ and $C=\{j\}$.

Now suppose $(i,j) \notin \dom(w)$.
Then $\rank w_{[k][l]}\geq \rank w_{[i][j]} \geq 1$ whenever $m\geq k\geq i$ and $n\geq l \geq j$.
It follows that $\MX_w$ contains the matrix $E_{ij}$ with $1$ in position $(i,j)$ and $0$ in all other positions, so $x_{ij} \notin I(\MX_w)$.
 It also follows that $x_{ij}$
does not occur as any of the monomials $\adiag(\cX_{RC})$ generating $J_w$ in \eqref{jdef-eq}, so $x_{ij} \notin J_w$ since $J_w$ is a monomial ideal.
\end{proof}

For positive integers $i<j$ let $t_{ij} = (i, j)\in S_\infty$ be the transposition swapping $i$ and $j$.
We write $u \lessdot v$ for $u,v \in S_\infty$ if $v$ covers $u$ in the Bruhat order on $S_\infty$.
Recall   from \cite[Lem.~2.1.4]{CCG} that this holds precisely when $v=ut_{ij}$ for integers $i<j$ with $u(i)<u(j)$ 
such that no $i<e<j$ has $u(i)<u(e)<u(j)$.

Suppose $0_{m\times n} \neq w \in S_\infty^{(m,n)}$ so that $I(\MX_w)$ is not a maximal ideal.
Then there exists an outer corner of $\dom(w)$ in $[m]\times [n]$.
Choose any such outer corner $(p,q) \in[m]\times [n]$ and set 
\be\label{msv-f-eq}
f = x_{pq} \in \KK[\Mat_{m\times n}].
\ee
Since $(p,q) \notin \dom(w)$, we have $f \notin I(\MX_w)$ by Proposition~\ref{exerprop}.
Next, consider the set 
\be
\cC(w) = \{ wt_{pr} \in S_\infty : p < r \in \PP\text{ and } w \lessdot wt_{pr} \}.
\ee
It is known that $\varnothing \subsetneq \cC(w) \subseteq S_\infty^{(m,n)}$ by \cite[Lem.~5.8]{MP2022}. Finally define 
\be\label{msv-phi-eq}
\Phi  = \{ ( I(\MX_v), J_v) : v \in \cC(w)\}.
\ee
To make these definitions canonical, one can take $(p,q)$ 
to be the lexicographically minimal outer corner of $\dom(w)$ in $[m]\times [n]$,
although any outer corner works just as well.

\begin{proof}[Proof of Theorem~\ref{msv-thm}]
We have already observed in Sections~\ref{msv-sect1} and \ref{msv-pp-sect} that axioms (T1) and (T2) in Definition~\ref{ts-def} hold.

For the other axioms, fix $ w \in S_\infty^{(m,n)}$ with $I(\MX_w)$ not maximal. 
Choose an outer corner $(p,q) \in [m]\times [n]$ of $\dom(w) $, and 
define $f = x_{pq}\notin I(\MX_w)$ and $\Phi \subseteq \cT$ as in \eqref{msv-f-eq} and \eqref{msv-phi-eq}.
The identities required for axiom (T3) when $(I,J)= (I(\MX_w),J_w)$ are provided in \cite[Lem.~5.11]{MP2022}.

It remains to check axiom (T4).
This holds by default as 
 $(I(\MX_w) \colon \langle f \rangle) =I(\MX_w)$. There is a simple geometric reason for this equality: if $H\subset \Mat_{m\times n}$ is the hyperplane where $f = 0$, then 
 \[(I(\MX_w) \colon \langle f \rangle) = I(\overline{\MX_w \setminus H}),\] and $\overline{\MX_w \setminus H} = \MX_w$ because $\MX_w$ is an irreducible variety with $\MX_w \not\subseteq H$ as $f \notin I(\MX_w)$. 

One can also show  $(I(\MX_w) \colon \langle f \rangle) =I(\MX_w)$ by an algebraic method based on Lemma~\ref{tlem1}.
 See the proof of Theorem~\ref{ss-msv-thm} for a prototype of this argument.
\end{proof}
 
The preceding theorem lets us recover some results of Knutson and Miller \cite{KnutsonMiller}.

 \begin{corollary}[\cite{KnutsonMiller}] 
 \label{msv-cor}
 If $w \in \Mat_{m\times n}$ then
$I_w = I(\MX_w)$ is prime and $J_w = \init(I_w)$,
and the set of minors generating $I_w$ in \eqref{msv-minors-eq} is a Gr\"obner basis in the reverse lexicographic term order.
\end{corollary}
 
 \begin{proof}
Corollary~\ref{ts-gr-cor} and Theorem~\ref{msv-thm}
imply that the generating set in \eqref{msv-minors-eq} is a Gr\"obner basis for $I(\MX_w)$.
Thus $I_w = I(\MX_w)$ 
is prime as $\MX_w$ is irreducible \cite[Ch.~15]{MillerSturmfels}, and $J_w = \init(I_w)$.
 \end{proof}

\section{Skew-symmetric matrix Schubert varieties}\label{schumat-sect2}

One can construct a similar transition system for the skew-symmetric analogues of $\MX_w$.
We describe this here.
This will generalize Example~\ref{ss-ts-ex}
and give another application of Theorem~\ref{new-ts-thm},
 recovering some results from \cite{MP2022}.

Recall that $\KK$ is an arbitrary field that is algebraically closed. We define a matrix $M$ to be \defn{skew-symmetric} if $M^\top=-M$ and all diagonal entries of $M$ are zero; such matrices are also often called \defn{alternating}.
Let
$ \SSMat_{n\times n}$ be the affine variety of skew-symmetric $n\times n$ matrices over  $\KK$.

\begin{definition}
For each skew-symmetric matrix $w \in \SSMat_{n\times n}$, the corresponding 
  \defn{skew-symmetric matrix Schubert cell}
  and  \defn{skew-symmetric matrix Schubert variety}
are the respective subsets
   \be \begin{aligned}
    \openSSX_w &= \left\{ M \in \SSMat_{n\times n} : \text{$\rank M_{[i][j]} = \rank w_{[i][j]}$ for all $(i,j) \in [n]\times [n]$}\right\},
    \\
    \SSX_w &= \left\{ M \in \SSMat_{n\times n} : \text{$\rank M_{[i][j]} \leq \rank w_{[i][j]}$ for all $(i,j) \in[n]\times [n]$}\right\},
\end{aligned}
\ee
where $M_{[i][j]}$ again stands for the upper left $i\times j$ submatrix of $M$. 
\end{definition}

Besides in \cite{MP2022},
the varieties $\SSX_w$ were previously studied in \cite{HMP6} and \cite{MP2020},
where their cohomology and $K$-theory classes were respectively computed (for relevant formulas, see also \cite{MP2021,MP2021b}).
The varieties $\SSX_w$ include a number of other families as special cases, like all rank $r$ skew-symmetric $n\times n$ matrices. Questions about ideals and Gr\"obner bases for some of these families were considered in various forms in earlier literature; we note in particular
\cite{DeNegri,DeNegriSbarra,HerzogTrung,JonssonWelker,RaghavanUpadhyay}.

We identify the coordinate ring of $\SSMat_{n\times n}$ with 
$\KK[\SSMat_{n\times n}] = \KK[u_{ij} : 1\leq j<i\leq n]$
where $u_{ij}$ denotes the linear function on matrices with $u_{ij}(M) = M_{ij} \in \KK$.
We order the monomials in this ring using the reverse lexicographic term order explained in Remark~\ref{rlto-rem}.

\subsection{Another reverse lexicographic transition system}\label{another-sect}

Fix a matrix $w \in \SSMat_{n\times n}$.
It is known from \cite{Cherniavsky} 
that
the variety $\SSX_w$ is irreducible and equal to the Zariski closure of  $\openSSX_w$.
The group $ \B_n$ of invertible lower-triangular $n\times n$ matrices over $\KK$ acts algebraically on $ \SSMat_{n\times n}$ by 
 $ g\cdot M = g M g^{\top}$, and the skew-symmetric matrix Schubert cell $\openSSX_w$ is  the orbit of $w$ under this action.
Thus, in the notation of Section~\ref{orbit-sect}, we can express
\be\label{ss-msv-m-eq} I(\SSX_w) = \cl_{\B_n}(\SScM_w)
\quad\text{for the maximal ideal }
\SScM_w := \langle u_{ij} - w_{ij} : 1\leq j < i \leq n\rangle.
\ee

Define $\SScX$
to be the $n\times n$ skew-symmetric matrix with entries $\SScX_{ij}  = u_{ij}=-\SScX_{ji}$ for $i>j$
and with $\SScX_{ii}=0$ for all $i$.
Then it is true, just as for ordinary matrix Schubert varieties, that
\be\label{ss-msv-minors-eq}
I(\SSX_w)\supseteq \left\langle \det(\SScX_{RC}) : i \in [n], \ j \in [n],\ R \subseteq [i],\ C\subseteq[j],\ |R|=|C|=1+\rank w_{[i][j]} \right\rangle.
\ee
This implies that, relative to the reverse lexicographic term order, one has 
{\small\be\label{ss-msv-minors-eq2}
\init(I(\SSX_w)) \supseteq \left\langle \adiag(\SScX_{RC}) : i \in [n], \ j \in [n],\ R \subseteq [i],\ C\subseteq[j],\ |R|=|C|=1+\rank w_{[i][j]} \right\rangle.
\ee}%
Unlike the story for ordinary matrix Schubert varieties, both of these containments can be strict \cite[Ex.~3.12]{MP2022}.
Therefore, to get a skew-symmetric version of Theorem~\ref{msv-thm}, we cannot use the right hand side of \eqref{ss-msv-minors-eq2} as a skew-symmetric analogue of $J_w$.

Instead, we consider a  more elaborate construction from \cite{MP2022}.
Suppose $A=\{a_0<a_1<\dots<a_r\}$ and $B = \{b_0<b_1<\dots<b_r\}$ are two sets of $r+1$ positive integers. Define 
\be A\odot B =\{(a_0,b_r),(a_1,b_{r-1}),\dots,(a_r,b_0)\}= \{(a_i,b_j) : i+j=r\}.\ee Then let 
$ A\boxplus B := \{ (i,j) \in A \odot B : i \geq j\} \cup \{ (i,j) \in B \odot A : i\geq j\}$
and set 
\be \textstyle \ssx_{AB} := \begin{cases} \prod_{(i,j) \in A\boxplus B} u_{ij} &\text{if every $(i,j) \in A\boxplus B$ has $i\neq j$} \\ 0 &\text{otherwise}.\end{cases} \ee
For example, if $a\neq b$ then  $\ssx_{\{a\}\{b\}} = u_{\max\{a,b\}\min\{a,b\}}$ while if  $A = \{1,3,4\}$ and $B = \{2,5,6\}$ then
$ 
A\boxplus B = \{(4,2),(5,3),(6,1)\}
$ and $ \ssx_{AB} = u_{42}u_{53}u_{61}.$
When nonzero, the monomial $\ssx_{AB}$ is always square-free. We record a lemma for later use:

\begin{lemma}\label{ssx-var-lem}
If $\ssx_{AB} = u_{ij}$ for some  $n \geq i >j \geq 1$ then $A=B = \{ i,j\}$ or $\{A,B\} = \{ \{i\},\{j\}\}$.
\end{lemma}

\begin{proof}
This can be shown in a self-contained way, but the elementary argument is fairly tedious. Alternatively, 
a special case of \cite[Lem~3.25]{MP2022}
asserts that $\ssx_{AB} $ divides $u_{ij}$ if and only if $\adiag(\cX_{AB})$ divides $x_{ij}x_{ji}$,
which occurs either when $A=B=\{i,j\}$ or
$\{A,B\} = \{ \{i\},\{j\}\}$.
\end{proof}

Now we introduce the monomial ideal
\be\label{ssj-def}
\SSJ_w = \left\langle \ssx_{AB} : i,j \in [n], \ i\geq j,\ A \subseteq [i],\ B\subseteq[j],\ |A|=|B|=1+\rank w_{[i][j]} \right\rangle.
\ee
Unlike for the ideals in Section~\ref{schumat-sect1}, it is not obvious that $\SSJ_w \subseteq \init(I(\SSX_w))$. Nevertheless:

\begin{theorem} \label{ss-msv-thm}
The set $\SScF=\{(I(\SSX_w),\SSJ_w) : w \in \SSMat_{n\times n}\}$
is a finite transition system.
\end{theorem}
 
We explain the proof of this theorem below, following a strategy similar to the one in Section~\ref{schumat-sect1}.

 \subsection{Finite indexing set}\label{ss-msv-fin-sect}
 
 If $w \in \SSMat_{n\times n}$ then   $\SSX_{w} = \SSX_{gwg^\top}$ and $\SSJ_{w} = \SSJ_{gwg^\top}$ for all $g \in \B_n$,
 so the set $\SScF$ has cardinality at most the number $\B_n$-orbits in $\SSMat_{n\times n}$.
 We recall a finite indexing set for these orbits. 
 
 Let $\Ifpf$ be the set of fixed-point-free bijections $w : \PP \to \PP$ with $w(n) = n-(-1)^n$ for all sufficiently large $n\gg0$.
Equivalently, this is the $S_\infty$-conjugacy class of the infinite product of cycles $1_\fpf := (1, 2)(3, 4)(5, 6)\cdots.$
The \defn{visible descent set} of
$w \in \Ifpf$ is 
\be\DesFPF(w) := \{ i \in\PP: w(i) > w(i+1) <i\}.\ee
For each integer $n \in \NN$  define
$        \Ifpf^{(n)} := \left\{w \in \Ifpf : \DesFPF(w) \subseteq [n]\right\}.$
The set $\Ifpf^{(n)}$ is finite with cardinality equal to the number of involutions in $S_n$ \cite[Prop.~2.9]{MP2022}.

For each $w \in \Ifpf$ define $\ss_n(w) \in \SSMat_{n\times n}$ to be the matrix whose entry in position $(i,j)$
is $1$ if $i<j=w(i)$, $-1$ if $i>j=w(i)$, or else zero.
Then each $\B_n$-orbit in $\SSMat_{n\times n}$ contains an element of the form $\ss_n(w)$ for a unique $w \in \Ifpf^{(n)}$ by results in  \cite{Cherniavsky};
 it is not hard to show that this holds even when $\KK$ is not algebraically closed. 
Given $w \in \Ifpf^{(n)}$ we are therefore motivated to define 
\be\openSSX_w := \openSSX_{\ss_n(w)},
\qquad 
\SSX_w := \SSX_{\ss_n(w)},
\qquand
\SSJ_w := \SSJ_{\ss_n(w)}.\ee
Observe that $\openSSX_w = \openMX_w \cap \SSMat_{n\times n}$ and $\SSX_w = \MX_w \cap \SSMat_{n\times n}$.
Finally, we can write
\[\SScF =\left\{(I(\SSX_w),\SSJ_w) : w \in \SSMat_{n\times n}\right\}=\left\{(I(\SSX_w),\SSJ_w) : w \in  \Ifpf^{(n)} \right\}\]
and the second set is uniquely indexed,
since if $v,w \in   \Ifpf^{(n)}$ then 
$I(\SSX_v) \subseteq I(\SSX_w)$ 
 if and only if $v\leq w$ in a certain \defn{Bruhat order} on $\Ifpf$ \cite[Prop.~3.11]{MP2022}.
 
\begin{example}
As in the classical case, the ideal $I(\SSX_w)$ is maximal if and only if $\SSX_w$ is a point, which occurs only when $\SSX_w = \{0\}$.
The index $w \in \Ifpf^{(n)}$ for this case is 
\[ 0^\ss_{n\times n} := (1, n+1)(2, n+2)\cdots(n,2n)(2n+1, 2n+2)(2n+3, 2n+4)\cdots \in \Ifpf^{(n)} \] 
and the ideals 
$I(\SSX_{0^\ss_{n\times n}}) = \SSJ_{0^\ss_{n\times n}} = \init(I(\SSX_{0^\ss_{n\times n}})) = \langle u_{ij} : i,j\in [n],\ i> j\rangle$ are all irrelevant.
\end{example}

\subsection{Proof of Theorem~\ref{ss-msv-thm}}\label{ss-msv-non-sect}

Recall the definition of $\dom(w)$ and  its set of outer corners from Section~\ref{msv-non-sect}.

 \begin{proposition}\label{ss-exerprop}
Suppose $w \in \Ifpf^{(n)}$ and $(i,j) \in [n]\times [n]$ has $i>j$.
Then $(i,j) \in \dom(w)$ if and only if $u_{ij} \in I(\SSX_w)$ if and only if $u_{ij} \in \SSJ_w$.
\end{proposition}

\begin{proof}
Our argument is only slightly more complicated than the proof of Proposition~\ref{exerprop}.
First, assume $(i,j) \in \dom(w)$. Then $\rank M_{[i][j]}=\rank w_{[i][j]}=0$
and so
$M_{ij}=0$ for all $M \in \SSX_w$, which means that $u_{ij} \in I(\SSX_w)$. We also have  $u_{ij}  \in \SSJ_w$ 
since $u_{ij} = \ssx_{AB}$ 
 for $A = \{i\}$ and $B=\{j\}$.

Suppose conversely that $(i,j) \notin \dom(w)$.
As the partial permutation matrix of $w$ is symmetric with zeros on the diagonal,
it follows that 
 $\rank w_{[k][l]}=\rank w_{[l][k]}\geq \rank w_{[i][j]} \geq 1$ whenever $n\geq k\geq i$ and $n\geq l \geq j$, and also that  $\rank w_{[k][l]}\geq \rank w_{[i][i]} \geq 2$ whenever $ \min\{k,l\}\geq i$.
One concludes that $\SSX_w= X_w \cap \SSMat_{n\times n}$ contains the matrix $E_{ij} - E_{ji}$ so $u_{ij} \notin I(\SSX_w)$.

To show that $u_{ij} \notin \SSJ_w$, we just need to check that $u_{ij}$ does not equal any of the terms $\ssx_{AB}$ generating $\SSJ_w$ in \eqref{ssj-def}.  
This follows from Lemma~\ref{ssx-var-lem}, as if $A\subseteq [k]$ and $B\subseteq [l]$
have $A=B=\{i,j\}$,
then $\min\{k,l\}\geq i$ so $\rank w_{[k][l]} \geq 2$, while if $\{A,B\}= \{\{i\},\{j\}\}$
then similarly $\rank w_{[k][l]}\geq 1$.
\end{proof}

Following \cite[\S4.1]{HMP3}, we define the \defn{fpf-involution length} of   $w \in \Ifpf$ to be
\be \ellfpf(w) := |\{ (i,j) \in \PP\times \PP : w(i) > w(j) < i <j\}|.\ee
Write $u \lessdot_\fpf v$ for $u,v \in \Ifpf$ if   $v=t_{ij}\cdot u\cdot t_{ij}$ for any positive integers $i<j$ with $\ellfpf(v) = \ellfpf(u)+1$.
For a description of $\lessdot_\fpf$ 
in terms of the cycles of $u$ and $v$, see \cite[Prop.~4.9]{HMP3}.

Suppose $0^\ss_{n\times n} \neq w \in \Ifpf^{(n)}$ so that $I(\SSX_w)$ is not a maximal ideal and $\SSX_w \neq \{0\}$.
Then  $\dom(w)$, which is invariant under transpose, cannot contain every position in  $[n]\times [n]$,
so there must exist an outer corner $(p,q)$ of $\dom(w)$ with $n \geq p\geq q \geq 1$. 
As explained in \cite[\S4.1]{MP2022}, all outer corners $(p,q)$ of $\dom(w)$ have $w(p)=q$, so as $w$ has no fixed points we can assume that $p> q$.

In summary, we can
choose an outer corner $(p,q)$ of $\dom(w)$ with $n\geq p>q\geq1$. 
Make such a choice (or to be canonical, let $(p,q)$ be the lexicographically minimal choice) and then set
\be\label{ss-msv-f-eq}
f= u_{pq} \in \KK[\SSMat_{n\times n}].
\ee
Since $(p,q) \notin \dom(w)$, we have $f \notin I(\SSX_w)$ by Proposition~\ref{ss-exerprop}.
Next, define the set 
\be
\cC^\ss(w) = \{ t_{pr}\cdot  w\cdot t_{pr} \in \Ifpf : p < r \in \PP\text{ and } w \lessdot_\fpf t_{pr}\cdot w\cdot t_{pr} \}.
\ee
It is known that $\varnothing \subsetneq \cC^\ss(w) \subseteq \Ifpf^{(n)}$ by \cite[Lem.~4.8]{MP2022}. Finally, let 
\be\label{ss-msv-phi-eq}
\Phi = \{ ( I(\SSX_v), \SSJ_v) : v \in \cC^\ss(w)\}.
\ee

\begin{proof}[Proof of Theorem~\ref{ss-msv-thm}]
The claim that $\SSJ_w \subseteq \init(I(\SSX_w))$
holds by \cite[Lem~3.37 and Thm.~3.17]{MP2022}. 
The other parts of axioms (T1) and (T2) in Definition~\ref{ts-def} were checked in  Sections~\ref{ss-msv-fin-sect} and \ref{ss-msv-non-sect}.

Fix $ w \in \Ifpf^{(n)}$ with $I(\SSX_w)$ not maximal, and then choose an outer corner $(p,q)$ of $\dom(w)$ with $n\geq p > q \geq 1$.
The identities required for axiom (T3) when $(I,J) = (I(\SSX_w),\SSJ_w)$ are supplied in \cite[Lem.~4.2]{MP2022}
if we define
 $f= u_{pq}\notin I(\SSX_w)$ and $\Phi$ as in \eqref{ss-msv-f-eq} and \eqref{ss-msv-phi-eq}.

It will turn out that axiom (T4) holds vacuously as
 $(I(\SSX_w) \colon \langle f \rangle) =I(\SSX_w)$.
To show this, we examine the group action of $G=\B_n$ on $A=\SSMat_{n\times n}$.
If we identify 
 \[\KK[\B_n] = \KK\left[b_{ij}, b_{ii}^{-1}: 1\leq j\leq i \leq n\right]\] where  $b_{ij}(g) = g_{ij}$,
then the map  $\alpha^\ast $  
from \eqref{pullback-eq}
has the formula
\be
\alpha^\ast(u_{ij})=  \sum_{\substack{p,q \in [n] \\ 1\leq q \leq j < p \leq i}}  b_{ip} b_{jq} \otimes u_{pq} + \sum_{\substack{p,q \in [n] \\ 1\leq q < p \leq j<i}}  (b_{ip} b_{jq}-b_{iq} b_{jp}) \otimes u_{pq}
\quad\text{for $1\leq j<i\leq n$.}\ee

The hypothesis of Lemma~\ref{tlem1}
holds for the partial order $\preceq$ on   $\cN:=\{ (i,j) \in [n]\times [n] : i>j\}$
with $(i,j) \preceq (k,l)$ if both $i\leq k$ and $j \leq l$. 
Since $(p,q)$ is an outer corner of $\dom(w)$, we have 
 $(i,j) \in \dom(w)$ for all $(i,j) \in \cN$ with $(i,j) \prec (p,q)$,
and so $u_{ij} \in I(\SSX_w)$ for all  $(i,j) \in \cN$ with $(i,j) \prec (p,q)$ by Proposition~\ref{ss-exerprop}.
Thus, referring to \eqref{ss-msv-m-eq},
we conclude by Lemma~\ref{tlem1}(a) that $u_{pq} \notin \SScM_w$ since $u_{pq} \notin I(\SSX_w)$. Finally, from
Lemmas~\ref{tlem1}(b) and \ref{tlem2}, we get 
\[(I(\SSX_w) \colon \langle f\rangle) = (\cl_{\B_n}(\SScM_w) \colon \langle u_{pq}\rangle) =  \cl_{\B_n}((\SScM_w \colon \langle u_{pq}\rangle ))
=   \cl_{\B_n}(\SScM_w )= I(\SSX_w)\]
as promised, and so axiom (T4) is automatically satisfied.
\end{proof}

 Theorem~\ref{ss-msv-thm} immediately recovers the following   property from \cite[Thm.~4.3]{MP2022}.

 \begin{corollary}[\cite{MP2022}] 
If $w \in \SSMat_{n\times n}$ 
then $\SSJ_w = \init(I(\SSX_w))$.
\end{corollary}

So far we have avoided defining a skew-symmetric analogue of the ideal $I_w$ from \eqref{msv-minors-eq}.
Such an ideal $\SSI_w \subseteq \KK[\SSMat_{n\times n}]$ is introduced for each $w \in  \Ifpf^{(n)}$ in 
\cite[Def.~3.13]{MP2022}. We omit an explicit description here, 
except to state that this ideal $\SSI_w$ can be generated a set of \defn{Pfaffians} of certain submatrices of the skew-symmetric matrix of variables $\SScX = \begin{bsmallmatrix} 0 & -u_{ij}  \\ u_{ij} & 0 \end{bsmallmatrix}$.

The simplest Pfaffian generating set for $\SSI_w$ given in \cite[Def.~3.13]{MP2022} is not a Gr\"obner basis \cite[Ex.~3.28]{MP2022}. However,
there is a more complicated set of Pfaffians  $G_w^\ss\subseteq \SSI_w$ described in \cite[Thm.~4.6]{MP2022}
for which it turns out that $\SSJ_w \subseteq \langle \init(g) : g \in G_w^\ss \rangle$ \cite[Lems.~3.35 and 3.36]{MP2022}.

It can be shown that $\SSX_w$ is the zero locus of $\SSI_w$ and so $  G_w^\ss \subseteq \SSI_w \subseteq I(\SSX_w)$ \cite[Thm.~3.17]{MP2022}. 
Therefore,   Corollary~\ref{ts-gr-cor} implies that  $G_w^\ss$ is a Gr\"obner basis for $I(\SSX_w)$, 
so $\SSI_w = \langle G_w^\ss \rangle = I(\SSX_w)$ is prime (as $\SSX_w$ is irreducible) and $\SSJ_w = \init(\SSI_w)$. 
These facts recover \cite[Thms.~4.5 and 4.6]{MP2022}.

\section{Symmetric matrix Schubert varieties}\label{schumat-sect3}

Replacing skew-symmetry $M=-M^\top$ by transpose invariance $M=M^\top$ yields symmetric versions of the objects in the previous two sections.
This section presents some conjectures about these objects, generalizing Example~\ref{22sym-ex}.
For $n \in \PP$ 
write
$ \SymMat_{n\times n}$ for the affine variety of symmetric $n\times n$ matrices over our arbitrary algebraically closed field $\KK$.

\begin{definition}\label{SymX-def}
For each symmetric matrix $w \in \SymMat_{n\times n}$, the corresponding 
  \defn{symmetric matrix Schubert cell}
  and  \defn{symmetric matrix Schubert variety}
are the respective subsets
   \be \begin{aligned}
    \openSymX_w &= \left\{ M \in \SymMat_{n\times n} : \text{$\rank M_{[i][j]} = \rank w_{[i][j]}$ for all $(i,j) \in [n]\times [n]$}\right\},
    \\
    \SymX_w &= \left\{ M \in \SymMat_{n\times n} : \text{$\rank M_{[i][j]} \leq \rank w_{[i][j]}$ for all $(i,j) \in[n]\times [n]$}\right\},
\end{aligned}
\ee
where as usual $M_{[i][j]}$ stands for the upper left $i\times j$ submatrix of $M$. 
\end{definition}

As with $\SSX_w$, the varieties $\SymX_w$ include many familiar classes of symmetric matrices,
whose ideals have been previously studied (see, e.g., \cite{Gorla,Kutz}). Questions about the Gr\"obner geometry of  symmetric matrix Schubert varieties are less well-understood.
We know of only a few relevant references \cite{Conca,Conca2,DeKr,GMN}, which supply answers in special cases.

\begin{remark*}
Ideals and Gr\"obner bases for a different collection of ``symmetric matrix Schubert varieties'' are considered in 
\cite{EFRW,FRS}.
The varieties in these references are defined by imposing northeast rank conditions,
while $\SymX_w$ is defined via northwest rank conditions. The two families are not generally related in any simple way.
\end{remark*}

We identify the coordinate ring of $\SymMat_{n\times n}$ with  
$\KK[\SymMat_{n\times n}] = \KK[u_{ij} : 1\leq j\leq i\leq n]$
where $u_{ij}$ represents the linear function on matrices with $u_{ij}(M) = M_{ij} \in \KK$.
We continue to order the monomials in this ring using the reverse lexicographic term order explained in Remark~\ref{rlto-rem}.
(Note that this is not the term order considered in \cite{Conca}.)

It is an open problem to determine the initial ideals and associated Gr\"obner bases
for $I(\SymX_w)$. Computations support a plausible conjecture.
Below, we discuss this conjecture along with a speculative proof strategy using transition systems. Turning this approach into a detailed proof will require new ideas to overcome obstacles that did not arise in Sections~\ref{schumat-sect1} or \ref{schumat-sect2}.

\subsection{An incomplete transition system}

Fix an element  $w \in \SymMat_{n\times n}$.
At least when $\ch(\KK) \neq2$, the variety $\SymX_w$ is  irreducible and  equal to the Zariski closure of  $\openSymX_w$ \cite[Lem.~5.2]{BagnoCherniavsky}. The reference \cite{BagnoCherniavsky} works over  $\CC$, but the relevant arguments hold over any algebraically closed field with $\ch(\KK)\neq2$. 

\begin{remark}
The lower-triangular Borel subgroup $ \B_n$ acts algebraically on $ \SymMat_{n\times n}$ by 
the same formula as in the skew-symmetric case: $ g\cdot M = g M g^{\top}$.
When $\ch(\KK)\neq 2$, the matrix Schubert cell
 $\openSymX_w$ is  the orbit of $w$ under this action,
and so in the notation of Section~\ref{orbit-sect}, we have
\be\label{sym-msv-m-eq} I(\SymX_w) = \cl_{\B_n}(\SymcM_w)
\quad\text{for the maximal ideal }
\SymcM_w := \langle u_{ij} - w_{ij} : 1\leq j \leq i \leq n\rangle.\ee
However, if $\ch(\KK)=2$ then $\openSymX_w$ may be a union of multiple of $\B_n$-orbits.
\end{remark}

Define $\SymcX$
to be the $n\times n$ symmetric matrix with entries $\SymcX_{ij}  = u_{ij}=\SymcX_{ji}$ for all $i\geq j$.
Similar to our two previous cases, if we consider the ideal generated by minors 
 \be\label{sym-msv-minors-eq}
\SymI_w:= \left\langle \det(\SymcX_{RC}) : i \in [n], \ j \in [n],\ R \subseteq [i],\ C\subseteq[j],\ |R|=|C|=1+\rank w_{[i][j]} \right\rangle
\ee
and also define 
\be\label{sym-msv-minors-eq2} \SymJ_w := \left\langle \adiag(\SymcX_{RC}) : i \in [n], \ j \in [n],\ R \subseteq [i],\ C\subseteq[j],\ |R|=|C|=1+\rank w_{[i][j]} \right\rangle\ee
then it automatically holds that
\be\label{sym-msv-minors-eq3}
\SymI_w \subseteq I(\SymX_w) \quand  \SymJ_w \subseteq \init(\SymI_w)\subseteq \init(I(\SymX_w))\ee relative to the reverse lexicographic term order.

 We expect that the  two containments in \eqref{sym-msv-minors-eq3} are actually both equalities (see Conjectures~\ref{sym-conj3a} and \ref{sym-conj3b}),
 just as in the ordinary matrix Schubert case. However, this cannot be shown by proving the most obvious symmetric reformulation of Theorem~\ref{msv-thm}, since the set 
\be\label{symf-eq} \{(I(\SymX_w),\SymJ_w) : w \in \SymMat_{n\times n}\}\ee
is generally  too small to be a transition system (or even a partial transition system), as we have already seen in Example~\ref{22sym-ex}.
Nevertheless, computations support the following conjecture:

\begin{conjecture}\label{sym-conj}
There exists a transition system containing 
$ \{(I(\SymX_w),\SymJ_w) : w \in \SymMat_{n\times n}\}$.
\end{conjecture}

We suspect, in view of Example~\ref{22sym-ex}, that the transition systems realizing Conjecture~\ref{sym-conj}
may have to include non-radical ideals of $\KK[\SymMat_{n\times n}]$.
We predict that these systems will still be finite, however.

One can see that the set $ \{(I(\SymX_w),\SymJ_w) : w \in \SymMat_{n\times n}\}$ is itself finite by 
noting that the set is unchanged when we restrict $w$ to range over all symmetric $n\times n$ partial permutation matrices,
since these matrices give rise to all possible 
 rank tables for elements of $\SymMat_{n\times n}$.
 The latter claim can be deduced (when $\KK$ is an arbitrary field) from \cite[Lem.~3.20]{MP2020},
 which shows that the rank table of any matrix coincides with the rank table of some partial permutation matrix,
 and the latter must be symmetric for its rank table to be symmetric.

We will index the symmetric $n\times n$ partial permutation matrices 
by setting
 \be\label{II-eq} \I_\infty:= \{ w \in S_\infty : w=w^{-1} \}
\quand \I_\infty^{(n)} := \{ w \in \I_\infty : \DesR(w)\subseteq [n]\} = \I_\infty \cap S_\infty^{(n,n)}.\ee
Passing to the $n\times n$ partial permutation matrix gives a bijection from  $\I_\infty^{(n)} $
to symmetric $n\times n$ partial permutation matrices. This lets us define $\SymX_w$, $\SymI_w$, and $\SymJ_w$ for $w \in  \I_\infty^{(n)}$. Then 
\be
\left\{(I(\SymX_w),\SymJ_w) : w \in \SymMat_{n\times n}\right\} = \left\{(I(\SymX_w),\SymJ_w) : w \in \I_\infty^{(n)} \right\},
\ee
and the second set is uniquely indexed
since if $v,w \in   \I_\infty^{(n)}$ then 
$I(\SymX_v) \subseteq I(\SymX_w)$ 
 if and only if $v\leq w$ in the Bruhat order of $S_\infty$ by \cite[Lem.~3.5]{BagnoCherniavsky}.

 \begin{example}
 The ideal $I(\SymX_w)$ is maximal if and only if $\SymX_w$ is a point, which occurs only when
$\SymX_w = \{0\}$.
The index $w \in \I_\infty^{(n)}$ for this case is $0_{n\times n} = (1, n+1)(2, n+2)\cdots(n,2n)$, 
and 
we have
$I(\SymX_{0_{n\times n}}) = \SymI_{0_{n\times n}}= \SymJ_{0_{n\times n}} = \init(I(\SymX_{0_{n\times n}})) = \langle u_{ij} : i,j\in [n],\ i\geq j\rangle.$
\end{example}
 
We mention some interesting consequences
of Conjecture~\ref{sym-conj}. First, via Theorem~\ref{new-ts-thm}, the conjecture would immediately imply that $\SymJ_w$ is the 
initial ideal of $I(\SymX_w)$. Then it would follow by Corollary~\ref{ts-gr-cor} that  the minors listed in \eqref{sym-msv-minors-eq}
form a Gr\"obner basis for $I(\SymX_w)$.
As these minors already generate $\SymI_w$,  this would prove  the following:

 \begin{conjecture}\label{sym-conj3a}
If $w \in \SymMat_{n\times n}$ then $\SymI_w=I(\SymX_w)$ and $\SymJ_w=\init(I(\SymX_w))$.
\end{conjecture}

The next conjecture would also be a consequence:

  \begin{conjecture}\label{sym-conj3b}
If $w \in \SymMat_{n\times n}$ is any symmetric matrix then  $\init(\SymI_w) = \SymJ_w$ and
the set of minors generating $\SymI_w$ in \eqref{sym-msv-minors-eq}
is a Gr\"obner basis in the reverse lexicographic term order.
\end{conjecture}

Finally, as $\SymX_w$ is irreducible, Conjecture~\ref{sym-conj3a} would imply this last property:
 
  \begin{conjecture}\label{sym-conj4}
If $w \in \SymMat_{n\times n}$ is any symmetric matrix then $\SymI_w$ is a prime ideal of $\KK[\SymMat_{n\times n}]$.
\end{conjecture}

With some computer assistance, we can directly construct transition systems verifying 
Conjecture~\ref{sym-conj} for $n\leq 4$. The relevant calculations are explained in Section~\ref{tf-sect}.
Using the computer algebra system  {\tt Macaulay2}, we have also been able to check that Conjecture~\ref{sym-conj3b} 
holds when $n\leq 7$ for any choice of field $\KK$.
We can also computationally verify
Conjecture~\ref{sym-conj4} 
when $n\leq 5$ and $\KK$ is any of the (not algebraically closed) fields $\QQ$, $\FF_2$, or $\FF_3$.

One application of the Gr\"obner basis computations
for matrix Schubert varieties and their skew-symmetric counterparts is to determine primary decompositions
of the initial ideals $J_w$ and $\SSJ_w$ for $I_w=I(\MX_w)$ and $\SSI_w =I(\SSX_w)$. These decompositions are intersections of monomial ideals indexed by certain \defn{pipe dreams} of $w$; see \cite[Thm.~B]{KnutsonMiller} and \cite[Thm.~4.15]{MP2020}.

We expect that results will hold for symmetric matrix Schubert varieties.
Define an \defn{involution pipe dream} for $z \in \I_\infty^{(n)}$ as in \cite[\S1.2]{HMP6}:
this is a subset of $ \ltriangeq_n := \{ (i,j) \in [n]\times [n]: i\geq j\}$
whose reading word, appropriately defined, determines a certain kind of ``almost symmetric'' reduced word for $z$.
Each involution pipe dream $D$ has an associated wiring diagram and one can define
$m_D : \ltriangeq_n \to \{1,2\}$ to be the map with $m_D(i,j)=2$ 
if and only if $i\neq j$ and some pair of wires labeled by $k$ and $z(k)$ cross at $(i,j)$.

Computations for $n\leq 10$ support the following conjecture.
Recall that $\I_n\subsetneq \I_\infty^{(n)}$ is the set of involutions in $S_n$, corresponding to the symmetric $n\times n$ partial permutation matrices of full rank.
\begin{conjecture}\label{pd-conj}
If $w \in \I_n$ then $ \SymJ_w= \bigcap_{D}  \IPDideal_D$ 
where 
$\IPDideal_D :=  \left\langle u_{ij}^{m_D(i,j)} : (i,j) \in D\right\rangle
$ and where 
$D$ runs over all involution pipe dreams for all $z \in \I_n $ with $z \geq w$ in Bruhat order. 
\end{conjecture}

Replacing $\I_n$ by $\I_\infty^{(n)}$ in this conjecture would give the most natural symmetric analogue of \cite[Thm.~4.15]{MP2022}.
However, this broader claim is false, and to decompose $ \SymJ_w$ when $w$
corresponds to a symmetric $n\times n$ partial permutation matrix that is not invertible, a more involved statement is required.

We also mention that when $w\in \I_n$, the  decomposition given in Conjecture~\ref{pd-conj}
 is very redundant. It suffices to let $D$ run over the minimal elements of the poset on involution pipe dreams for $z \geq w$ defined by $D \leq E$ if $\IPDideal_D \subseteq \IPDideal_{E}$. If the exponents $m_D(i,j)$ in the definition of $\IPDideal_D$ were replaced by $1$, then these minimal elements would be just the involution pipe dreams for $w$. This does seem to hold if $w$ is \defn{non-crossing} in the sense of being $3412$-avoiding, but is not true in general.

\subsection{Transition forests}\label{tf-sect}

One can attempt to find a transition system containing a given family of ideals by recursively determining primary decompositions in the following way. This method is amenable to computer calculations using algebra systems like {\tt Macaulay2}.

For a non-maximal ideal $I\subset \KK[\SymMat_{n\times n}]$, let $u_I$ be the lexicographically smallest variable $u_{ij} \notin I$ 
and define $\sK_I$ to be any set of primary ideals such that
$I + \langle u_{ij}\rangle = \bigcap_{K \in \sK_I} K$ is a primary decomposition.
When the ideal $I$ is maximal, we leave $u_I$ undefined and set $\sK_I = \varnothing$.

Now, from a given set of ideals $\sI$, we form $\cT_\sI := \cT_\sI^0 \cup \cT_\sI^1 \cup   \cT_\sI^2 \cup \cdots$ where
\[\textstyle \cT_\sI^0 := \{ (I,\init(I)) : I \in \sI\} \quand \cT_\sI^{i+1} := \left\{ (K, \init(K)) : K \in \bigcup_{(I,J) \in \cT_\sI^{i}} \sK_I\right\}.
\]
The set $\cT_{\sI}^i$ is empty for large enough $i$,  
since the variable 
$u_I$ precedes $u_K$ in lexicographic order for all non-maximal ideals $K \in \sK_I$.

It is not guaranteed that this process will yield a family satisfying every part of Definition~\ref{ts-def}:
 problems can occur with conditions (T3) and (T4).
However, the construction of $\cT_\sI$ sometimes does produce a transition system.

We visualize $\cT_\sI$ by drawing a forest graph whose vertices are the ideals $I$ with $(I,J) \in \cT_{\sI}$. In this \defn{transition forest}, the children of a vertex labeled $I$ are the elements of $\sK_I$.
Although we have formulated these definitions specifically for the case $A=\SymMat_{n\times n}$, the relevant ideas easily extend to the case when $A$ is any affine space.

\begin{example}
We revisit the ideals $\sI = \left\{ I(\SymX_w) : w \in \I_\infty^{(2)}\right\}$ of $\KK[\SymMat_{2\times 2}]$ from Example~\ref{22sym-ex}.
If we define $I_1,\dots,I_6$ as in that example, then $\sI = \{I_i : i \in[5]\}$
and the transition forest of $\cT_\sI$ is
\[
\begin{tikzpicture}
\node {$I_{1}=\left\langle{u}_{21}^{2}-{u}_{11}{u}_{22}\right\rangle$}
    child {
        node {$\boxed{I_{6}=  \left\langle {u}_{11},\ {u}_{21}^{2}\right\rangle}$}
            child {
                node {$I_{4}= \left\langle{u}_{21},\ {u}_{11}\right\rangle$}
                    child {
                        node {$I_{5}= \left\langle {u}_{22},\ {u}_{21},\ {u}_{11}\right\rangle$}
                    }
            }
    }
;
\end{tikzpicture}
\quad
\begin{tikzpicture}
\node {$I_{2}=0$}
    child {
        node {$I_{3}=\left\langle {u}_{11}\right\rangle$}
            child {
                node {$I_{4}=\left\langle{u}_{21},\ {u}_{11}\right\rangle$}
                    child {
                        node {$I_{5}= \left\langle {u}_{22},\ {u}_{21},\ {u}_{11}\right\rangle$}
                    }
            }
    }
;
\end{tikzpicture}
\]
Here and in later pictures, we place boxes around the ideals $\boxed{I}$ with $I\notin \sI$.
As we know from Example~\ref{22sym-ex}, the set $\cT_\sI$ is a transition system
and therefore verifies Conjecture~\ref{sym-conj} when $n=2$.
\end{example}

  Recall that the \defn{associated primes} of an ideal $I$ are the prime ideals $\left\{\sqrt{K_1},\sqrt{K_2}, \ldots, \sqrt{K_m}\right\}$ where $I = K_1 \cap K_2\cap \cdots \cap K_m$ is a primary decomposition. An \defn{embedded prime} is an associated prime which is not minimal with respect to set-theoretic inclusion. If there are no embedded primes, then there is a  unique primary decomposition $I = K_1 \cap K_2\cap \cdots \cap K_m$ in which no $K_i$ is redundant and with all $\sqrt{K_1}, \sqrt{K_2},\ldots, \sqrt{K_m}$ distinct.
  
  If no embedded primes occur for  the ideals $I+\langle u_I\rangle$ with $(I,J) \in \cT_\sI$,
  then the set $\cT_\sI$ is uniquely determined as long as we insist that each $\sK_I$ be a unique irredundant primary decomposition.
In general, however, there can be multiple ways of constructing $\cT_\sI$ from a given set of ideals $\sI$.
This phenomenon did not arise in the previous example but does in the following one.

\begin{example}\label{sym3-ex}
Consider the ideals $\sI = \left\{ I(\SymX_w) : w \in \I_\infty^{(3)}\right\}$ of $\KK[\SymMat_{3\times 3}]$.
In this case $\sI$ has 14 elements $I_1,I_2,\dots,I_{14}$, which are indicated in Figure~\ref{i14-fig}. We have used {\tt Macaulay2} to find one valid construction of $\cT_\sI$. This requires adding  7 more ideals $I_{15},I_{16},\dots,I_{21}$, also shown in Figure~\ref{i14-fig}. Figure~\ref{f14-fig} displays the associated transition forest.

Embedded primes occur in this example, so $\cT_\sI$ is not unique. Specifically, for 
\[
I_{19} = \left\langle{u}_{11},\ {u}_{32}^{2}-{u}_{22}{u}_{33},\ {u}_{31}{u}_{32}-{u}_{21}{u}_{33},\ {u}_{31}^{2},\ {u}_{22}{u}_{31}-{u}_{21}{u}_{32},\ {u}_{21}{u}_{31},\ {u}_{21}^{2}\right\rangle
\]
we have a primary decomposition $I_{19} + \langle u_{21}\rangle = I_{11} \cap I_{21}$ where
\[
I_{11} =  \left\langle{u}_{31},\ {u}_{21},\ {u}_{11},\ {u}_{32}^{2}-{u}_{22}{u}_{33}\right\rangle
\quand
I_{21} = \left\langle{u}_{22},\ {u}_{21},\ {u}_{11},\ {u}_{32}^{2},\ {u}_{31}{u}_{32},\ {u}_{31}^{2}\right\rangle.
\]
However, it holds that
$
    \sqrt{I_{11}} = I_{11} = \langle u_{31}, u_{21}, u_{11}, u_{22}u_{33}-u_{32}^2\rangle \subsetneq \sqrt{I_{21}} = \langle u_{22}, u_{21}, u_{11}, u_{31}, u_{32}\rangle.
$
Regardless, we have checked that the construction of $\cT_\sI$ corresponding to what is shown in Figure~\ref{f14-fig} is a transition system. This verifies Conjecture~\ref{sym-conj} when $n=3$.
\begin{figure}[h]
\begin{center}
\begin{tabular}{c| l}
$w \in \I_\infty^{(3)}$ & $I(\SymX_w)\subset\KK[\SymMat_{3\times 3}]$  
\\[-8pt]\\ \hline \\[-8pt]
$\left[\begin{smallmatrix} 1 & 0 & 0 \\ 0 & 1 & 0 \\ 0 & 0 & 1 \end{smallmatrix}\right]$ & $I_1=0$
\\[-8pt]\\
$\left[\begin{smallmatrix} 1 & 0 & 0 \\ 0 & 1 & 0 \\ 0 & 0 & 0 \end{smallmatrix}\right]$ & $I_{2} = \left\langle{u}_{22}{u}_{31}^{2}-2{u}_{21}{u}_{31}{u}_{32}+{u}_{11}{u}_{32}^{2}+{u}_{21}^{2}{u}_{33}-{u}_{11}{u}_{22}{u}_{33}\right\rangle$
\\[-8pt]\\
$\left[\begin{smallmatrix} 1 & 0 & 0 \\ 0 & 0 & 1 \\ 0 & 1 & 0 \end{smallmatrix}\right]$ & $I_{3} = \left\langle{u}_{21}^{2}-{u}_{11}{u}_{22}\right\rangle$
\\[-8pt]\\
$\left[\begin{smallmatrix} 1 & 0 & 0 \\ 0 & 0 & 0 \\ 0 & 0 & 1 \end{smallmatrix}\right]$ & $I_{4} = \left\langle{u}_{22}{u}_{31}-{u}_{21}{u}_{32},\ {u}_{21}{u}_{31}-{u}_{11}{u}_{32},\ {u}_{21}^{2}-{u}_{11}{u}_{22}\right\rangle$
\\[-8pt]\\
$\left[\begin{smallmatrix} 1 & 0 & 0 \\ 0 & 0 & 0 \\ 0 & 0 & 0 \end{smallmatrix}\right]$ & $I_{5} = \left\langle
\barr{lll} 
{u}_{32}^{2}-{u}_{22}{u}_{33},&
{u}_{31}{u}_{32}-{u}_{21}{u}_{33},&
{u}_{31}^{2}-{u}_{11}{u}_{33},\\ 
{u}_{22}{u}_{31}-{u}_{21}{u}_{32},&
{u}_{21}{u}_{31}-{u}_{11}{u}_{32},&
{u}_{21}^{2}-{u}_{11}{u}_{22}
\earr\right\rangle$
\\[-8pt]\\
$\left[\begin{smallmatrix} 0 & 1 & 0 \\ 1 & 0 & 0 \\ 0 & 0 & 1 \end{smallmatrix}\right]$ & $I_{6}  = \left\langle{u}_{11}\right\rangle$
\\[-8pt]\\
$\left[\begin{smallmatrix} 0 & 1 & 0 \\ 1 & 0 & 0 \\ 0 & 0 & 0 \end{smallmatrix}\right]$ & $I_{7} = \left\langle{u}_{11},\ {u}_{22}{u}_{31}^{2}-2{u}_{21}{u}_{31}{u}_{32}+{u}_{21}^{2}{u}_{33}\right\rangle$
\\[-8pt]\\
$\left[\begin{smallmatrix} 0 & 0 & 1 \\ 0 & 1 & 0 \\ 1 & 0 & 0 \end{smallmatrix}\right]$ & $I_{8} = \left\langle{u}_{21},\ {u}_{11}\right\rangle$
\\[-8pt]\\
$\left[\begin{smallmatrix} 0 & 0 & 1 \\ 0 & 0 & 0 \\ 1 & 0 & 0 \end{smallmatrix}\right]$
&
$I_{9} = \left\langle{u}_{22},\ {u}_{21},\ {u}_{11}\right\rangle$
\\[-8pt]\\
$\left[\begin{smallmatrix} 0 & 0 & 0 \\ 0 & 1 & 0 \\ 0 & 0 & 1 \end{smallmatrix}\right]$ & $I_{10} = \left\langle{u}_{31},\ {u}_{21},\ {u}_{11}\right\rangle$
\\[-8pt]\\
$\left[\begin{smallmatrix} 0 & 0 & 0 \\ 0 & 1 & 0 \\ 0 & 0 & 0 \end{smallmatrix}\right]$ & $I_{11} =  \left\langle{u}_{31},\ {u}_{21},\ {u}_{11},\ {u}_{32}^{2}-{u}_{22}{u}_{33}\right\rangle$
\\[-8pt]\\
$\left[\begin{smallmatrix} 0 & 0 & 0 \\ 0 & 0 & 1 \\ 0 & 1 & 0 \end{smallmatrix}\right]$ & $I_{12} =\left\langle{u}_{31},\ {u}_{22},\ {u}_{21},\ {u}_{11}\right\rangle$
\\[-8pt]\\
$\left[\begin{smallmatrix} 0 & 0 & 0 \\ 0 & 0 & 0 \\ 0 & 0 & 1 \end{smallmatrix}\right]$ & $I_{13}  = \left\langle{u}_{32},\ {u}_{31},\ {u}_{22},\ {u}_{21},\ {u}_{11}\right\rangle$
\\[-8pt]\\
$\left[\begin{smallmatrix} 0 & 0 & 0 \\ 0 & 0 & 0 \\ 0 & 0 & 0 \end{smallmatrix}\right]$  & $I_{14} = \left\langle{u}_{33},\ {u}_{32},\ {u}_{31},\ {u}_{22},\ {u}_{21},\ {u}_{11}\right\rangle$
\\[-8pt]\\
 &$I_{15} = \left\langle{u}_{21},\ {u}_{11},\ {u}_{31}^{2}\right\rangle$
\\[-8pt]\\
 &$I_{16} = \left\langle{u}_{11},\ {u}_{21}^{2}\right\rangle$
\\[-8pt]\\
 &$I_{17} = \left\langle{u}_{11},\ {u}_{31}^{2},\ {u}_{22}{u}_{31}-{u}_{21}{u}_{32},\ {u}_{21}{u}_{31},\ {u}_{21}^{2}\right\rangle$
\\[-8pt]\\
 &$I_{18} = \left\langle{u}_{22},\ {u}_{21},\ {u}_{11},\ {u}_{31}^{2}\right\rangle$
\\[-8pt]\\
 &$I_{19} = \left\langle{u}_{11},\ {u}_{32}^{2}-{u}_{22}{u}_{33},\ {u}_{31}{u}_{32}-{u}_{21}{u}_{33},\ {u}_{31}^{2},\ {u}_{22}{u}_{31}-{u}_{21}{u}_{32},\ {u}_{21}{u}_{31},\ {u}_{21}^{2}\right\rangle$
\\[-8pt]\\
 &$I_{20} = \left\langle{u}_{31},\ {u}_{22},\ {u}_{21},\ {u}_{11},\ {u}_{32}^{2}\right\rangle$
\\[-8pt]\\
 &$I_{21} = \left\langle{u}_{22},\ {u}_{21},\ {u}_{11},\ {u}_{32}^{2},\ {u}_{31}{u}_{32},\ {u}_{31}^{2}\right\rangle$
\end{tabular}
\end{center}
\caption{Ideals of $\KK[\SymMat_{3\times 3}]$ defined for the transition system $\cT_\sI$ in Example~\ref{sym3-ex}}\label{i14-fig}
\end{figure}

\begin{figure}[h]
\[
\begin{tikzpicture}
\node {$I_{1}$}
    child {
        node {$I_{6}$}
            child {
                node {$I_{8}$}
                    child {
                        node {$I_{10}$}
                            child {
                                node {$I_{12}$}
                                    child {
                                        node {$I_{13}$}
                                            child {
                                                node {$I_{14}$}
                                            }
                                    }
                            }
                    }
            }
    }
;
\end{tikzpicture}
\quad
\begin{tikzpicture}
\node {$I_{2}$}
    child {
        node {$I_{7}$}
            child {
                node {$I_{9}$}
                    child {
                        node {$I_{12}$}
                            child {
                                node {$I_{13}$}
                                    child {
                                        node {$I_{14}$}
                                    }
                            }
                    }
            }
            child {
                node {$\boxed{I_{15}}$}
                    child {
                        node {$I_{10}$}
                            child {
                                node {$I_{12}$}
                                    child {
                                        node {$I_{13}$}
                                            child {
                                                node {$I_{14}$}
                                            }
                                    }
                            }
                    }
            }
    }
;
\end{tikzpicture}
\quad
\begin{tikzpicture}
\node {$I_{3}$}
    child {
        node {$\boxed{I_{16}}$}
            child {
                node {$I_{8}$}
                    child {
                        node {$I_{10}$}
                            child {
                                node {$I_{12}$}
                                    child {
                                        node {$I_{13}$}
                                            child {
                                                node {$I_{14}$}
                                            }
                                    }
                            }
                    }
            }
    }
;
\end{tikzpicture}
\quad
\begin{tikzpicture}
\node {$I_{4}$}
    child {
        node {$I_{9}$}
            child {
                node {$I_{12}$}
                    child {
                        node {$I_{13}$}
                            child {
                                node {$I_{14}$}
                            }
                    }
            }
    }
    child {
        node {$\boxed{I_{17}}$}
            child {
                node {$I_{10}$}
                    child {
                        node {$I_{12}$}
                            child {
                                node {$I_{13}$}
                                    child {
                                        node {$I_{14}$}
                                    }
                            }
                    }
            }
            child {
                node {$\boxed{I_{18}}$}
                    child {
                        node {$I_{12}$}
                            child {
                                node {$I_{13}$}
                                    child {
                                        node {$I_{14}$}
                                    }
                            }
                    }
            }
    }
;
\end{tikzpicture}
\quad
\begin{tikzpicture}
\node {$I_{5}$}
    child {
        node {$\boxed{I_{19}}$}
            child {
                node {$I_{11}$}
                    child {
                        node {$\boxed{I_{20}}$}
                            child {
                                node {$I_{13}$}
                                    child {
                                        node {$I_{14}$}
                                    }
                            }
                    }
            }
            child {
                node {$\boxed{I_{21}}$}
                    child {
                        node {$\boxed{I_{20}}$}
                            child {
                                node {$I_{13}$}
                                    child {
                                        node {$I_{14}$}
                                    }
                            }
                    }
            }
    }
;
\end{tikzpicture}
\]
\caption{Transition forest constructed for Example~\ref{sym3-ex}. The boxed ideals $\boxed{I}$ are the ideals not of the form $I(\SymX_w)$ for any $w \in \SymMat_{3\times 3}$.}\label{f14-fig}
\end{figure}
\end{example}

\begin{example}
The family of ideals $\sI = \left\{ I(\SymX_w) : w \in \I_\infty^{(4)}\right\}$ in $\KK[\SymMat_{4\times 4}]$ has 43 elements. We have used {\tt Macaulay2} to construct $\cT_\sI$ and check that it is a transition system, verifying Conjecture~\ref{sym-conj} when $n=4$.
Our calculation gives a set $\cT_\sI$ with 86 elements. The corresponding transition forest is too large to show here, but we mention that it has vertices with $\geq 2$ children.
\end{example}

There is a weaker form of Conjecture~\ref{sym-conj} that we can check by computer in a few more cases.
Let $\I_n = \{ w \in S_n : w=w^{-1} \}$ be the subset of elements in $\I_\infty^{(n)}$ whose $n\times n$ permutation matrices have full rank.
Using the same algorithm as in the preceding examples, we have been able to construct 
a transition system   containing $ \{(I(\SymX_w),\SymJ_w) : w \in \I_n\}$ for all $n\leq 6$.

\section{Applications to stable limits}\label{stable-sect}

In this final section we study the \defn{$K$-polynomials} of the varieties of the ideals
 $\SymI_w\subset \KK[\SymMat_{n\times n}]$ defined in \eqref{sym-msv-minors-eq}. 
 Recall that these ideals are indexed by 
 involutions $w$ in the set  $\I_\infty^{(n)}$ from \eqref{II-eq}. 
 
 If we assume Conjecture~\ref{sym-conj3a}, then 
the $K$-polynomials of the determinantal ideals $\SymI_w$ coincide after a change of variables with the family of
\defn{orthogonal Grothendieck polynomials} introduced in \cite{MP2020}.
 After reviewing some relevant definitions, 
 we will prove that if Conjecture~\ref{sym-conj3a} holds then each of these polynomials has a well-defined ``stable limit'' in the ring of formal power series.
This will connect the main conjectures in Section~\ref{schumat-sect3} to an open problem posed in \cite{MP2020}. 

 \subsection{Equivariant K-theory}
 
We first recall a few facts about equivariant $K$-theory from \cite[\S5]{ChGi}.  Let $R$ be a
finitely generated commutative $\KK$-algebra
such that
$X = \Spec R$ is a smooth affine variety. Suppose that a linear algebraic group $G$ acts on $X$ algebraically. The action of $G$ on $X$ makes $R$ a comodule for the Hopf algebra $\KK[G]$. An algebraic $R$-module $M$ is an \defn{$(R,G)$-module} if it also has the structure of a $\KK[G]$-comodule in such a way that the multiplication map $R \otimes M \to M$ is a morphism of $\KK[G]$-comodules. 

The \defn{equivariant $K$-theory group $K_G(X)$} is the Grothendieck group of the category of finitely generated $(R,G)$-modules. We usually write $[M]$ for the class of $M$ in $K_G(X)$, but to emphasize the group $G$ we sometimes express this as
$[M]_G$.  When $Z$ is a subscheme of $X$ with ideal $I(Z)$, we define $[Z] := [R/I(Z)]$. Here are some relevant observations:

\begin{itemize}
  \item We can take the tensor product over $R$ of $(R,G)$-modules, but this does not yield a well-defined product on $K_G(X)$ because tensor products need not preserve exact sequences. Correcting this by defining $[M][N] = \sum_{i \geq 0} (-1)^i[\Tor_i^R(M,N)]$ makes $K_G(X)$ into a commutative ring. It is a nontrivial fact that the smoothness of $X$ guarantees finiteness of the sum.
  
  \item Suppose $j : Z \hookrightarrow X$ is the inclusion of a smooth $G$-invariant subvariety. If $M$ is an $(R,G)$-module then each $\Tor_i^R(R/I(Z),M)$ is an $(R/I(Z),G)$-module, so we can view $[M][Z]$ as an element of $K_G(Z)$. This defines a pullback homomomorphism $j^* :  K_G(X) \to K_G(Z)$.
  
  \item Suppose $\pi : \Spec S = Y \to X$ is a $G$-equivariant morphism. Then $S$ itself is an $(R,G)$-module via $\pi^* : R \to S$, and if $M$ is an $(R,G)$-module then $S \otimes_R M$ is an $(S,G)$-module. If $\pi$ is a \defn{flat} morphism, meaning that $S$ is a flat $R$-module, then  $ S\otimes_R \bullet$ is an exact functor from $(R,G)$-modules to $(S,G)$-modules, and one can define a pullback morphism $\pi^* : K_G(X) \to K_G(Y)$. If $\KK = \CC$ and $\pi$ is a homotopy equivalence, then $\pi^*$ is an isomorphism.
  
  \item A $(\KK,G)$-module is just a representation of $G$, so $K_G(\pt)$ is the Grothendieck ring $\Rep(G)$ of finite-dimensional $G$-representations. The pullback of the constant map $X \to \pt$ gives $K_G(X)$ the structure of a $\Rep(G)$-algebra.
\end{itemize}

 We note one general fact for use in the next section.
 
 \begin{lemma} \label{lem:subrep-pullback}
    Suppose $V = U \oplus U'$ is a finite-dimensional representation of $G$, with $U$ and $U'$ subrepresentations and $\alpha : U \hookrightarrow V$ the inclusion. Then $\Rep(G) \cong K_G(V) \xrightarrow{\alpha^*} K_G(U) \cong \Rep(G)$ is the identity map.
\end{lemma}

\begin{proof}
    Let $\pi : V \to U$ be the natural projection map. As $[M] \in \Rep(G)$ is identified with $[\KK[U] \otimes_{\KK} M] \in K_G(U)$, we have
    $
            \pi^*[\KK[U] \otimes_{\KK} M] = [\KK[V] \otimes_{\KK[U]}  (\KK[U] \otimes M)] = [\KK[V] \otimes M],
$
    so $\pi^* : K_G(U) \to K_G(V)$ is the identity map. But $\alpha^* \pi^* = (\pi \circ \alpha)^* = \id$.
\end{proof}

Choose a torus $T\subseteq G$ and 
suppose $M$ is a representation of $T$ with a finite-dimensional weight decomposition, so that $M = \bigoplus_{\chi \in \Hom(T,\KK^\times)} M_{\chi}$ where each \defn{weight space} \[M_{\chi} = \{m \in M : tm = \chi(t)m\}\] is finite-dimensional. A \defn{weight vector} for $M$ is just an element of some $M_{\chi}$, and if $v \in M_{\chi}$ then we write $\wght(v) = \chi$. The \defn{character} of $M$ is then defined to be 
\[ \textstyle \ch(M) = \sum_{\chi \in \Hom(T,\KK^\times)} \dim(M_{\chi})\chi.\]
We view this as an element of the completion of the group ring $\ZZ[\Hom(T,\KK^\times)]$ when there are infinitely many nonzero summands.

Fix an isomorphism $T \cong (\KK^\times)^n$.  Then we can identify $\ZZ[\Hom(T,\KK^\times)] \cong \ZZ[a_1^{\pm 1}, a_2^{\pm1},\ldots, a_n^{\pm 1}]$, and under this identification $\ch(M)$ becomes a formal Laurent series in $a_1,a_2, \ldots, a_n$. 
An equivalent point of view is that $M$ is a vector space graded by the abelian group $\Hom(T,\KK^\times) \cong \ZZ^n$, with the piece in degree $\chi$ being the weight space $M_{\chi}$. For this reason, the character $\ch(M)$ is also called the \defn{multigraded Hilbert series} of $M$.

 If $X = \Spec R$ has a $T$-action, then $R$ itself is a $T$-module. An $(R,T)$-module $M$ is then the same thing as an $R$-module that is $\Hom(T,\KK^\times)$-graded in the sense that $R_\chi M_\psi \subseteq M_{\chi\psi}$.

\begin{definition}[{See \cite[Def.~8.21]{MillerSturmfels}}]
  If $V$ is a finite-dimensional representation of $T$, then the \defn{$K$-polynomial} of a finitely generated $(\KK[V], T)$-module $M$ is
$  \cK(M) = \frac{\ch( M)}{\ch( \KK[V])}.
$
\end{definition}

The $K$-polynomial $\cK(M)$ belongs \emph{a priori} to the field of formal Laurent series in $a_1,a_2, \ldots, a_n$.
It is known  \cite[Thm.~8.20]{MillerSturmfels} that this series  actually gives a polynomial
\be
\cK(M) \in \ZZ[a_1,a_2,\dots,a_n].
\ee
 This polynomial exactly describes the class of $M$ in $K_{T}(V)$, as we now explain. We have 
 \[K_{T}(\pt) = \Rep(T) \cong \ZZ[a_1^{\pm 1},\ldots, a_n^{\pm 1}]\] where the last isomorphism is the character map. Let $\pi : V \to \{0\}$ be the constant map. By definition, $\pi^*(f) = [\KK[V] \otimes_{\KK} W]$ for a $T$-module $W$ with character $f$. Thus $\ch(\pi^*(f)) = \ch(\KK[V])\ch(W)$, so we can recover $f$ from $\pi^*(f)$ as $\cK(\pi^*(f))$. 

At least over $\CC$, one can argue that  $\pi^*$ is a ring isomorphism because $\pi$ is a $T$-equivariant homotopy equivalence. Alternatively, the surjectivity of $\pi^*$ for general $\KK$ follows from the algebraic fact
\cite[Prop.~8.18]{MillerSturmfels} that every finitely generated $\ZZ^n$-graded $\KK[V]$-module has a finite resolution by graded free $\KK[V]$-modules. To sum up, sending $[M]_{T} \mapsto \cK(M)$ defines a ring isomorphism \[K_{T}(V) \xrightarrow{\sim} \ZZ[a_1^{\pm 1},a_2^{\pm 1}, \ldots, a_n^{\pm 1}].\]  
 We identify $K_{T}(V)$ with the ring of Laurent polynomials $\ZZ[a_1^{\pm 1},a_2^{\pm 1}, \ldots, a_n^{\pm 1}]$ via this map.

\begin{example}\label{22-k-ex}
Let $V = \SymMat_{2\times 2}$ and suppose $Z\subset V$ is the subvariety of noninvertible matrices (denoted $X_1$ in Example~\ref{22sym-ex}). 
Relative to the action $b : M \mapsto bMb^\top$,
the $\KK$-vector space $V$ is a finite-dimensional representation of the group $G=\B_2$ 
and therefore also of its torus $T$, which consists of all invertible $2\times 2$ diagonal matrices over $\KK$.
Let $a_i : T \to \KK$ for each $i \in \{1,2\}$ denote the representation $a_i(t) = t_{ii}$.

In the notation of Example~\ref{22sym-ex}
 we have 
$
 I(Z) = \langle u_{21}^2 - u_{11}u_{22}\rangle $
 and
$ [Z]_{T} = [\KK[V]/I(Z)]_{T}.$ Let $\cM$ be the set of all monomials in the variables $\{u_{11}, u_{21}, u_{22}\}$. These monomials are weight vectors for $T$ with weights $\wght(u_{ij}) = a_{i} a_j$   and $\wght(m_1 m_2) = \wght(m_1)\wght(m_2)$, and we have 
 \[\textstyle \ch(\KK[V]) = \sum_{m \in \cM} \wght(m).\] If we further let $\cM' \subseteq \cM$ be the subset of monomials divisible by $u_{21}^2$, then $\cM \setminus \cM'$ is a basis of weight vectors for $\KK[V]/I(Z)$. Therefore, we have
\begin{align*}
  \cK(Z)  := \cK(V/I(Z)) = \tfrac{\ch(\KK[V]/I(Z))}{\ch(\KK[V])} = \tfrac{\sum_{m \in \cM} \wght(m) - \sum_{m \in \cM'} \wght(m)}{\sum_{m \in \cM} \wght(m)} = 1 - \tfrac{\sum_{m \in \cM'} \wght(m)}{\sum_{m \in \cM} \wght(m)}.
\end{align*}
The map $m \mapsto u_{21}^2 m$ is a bijection $\cM \to \cM'$ with $\wght(u^2_{21}m) = (a_1 a_2)^2 \wght(m)$, so in fact 
\[\cK(Z) =1-(a_1 a_2)^2.\]

If $v_1, v_2,\ldots, v_m$ is a basis of weight vectors for any $T$-representation $V$, then one always has the explicit formula $\ch (\KK[V]) = \prod_{i=1}^m \tfrac{1}{1-\wght(v_i)}$. However, as this example suggests, this computation is not really necessary for determining $K$-polynomials.
\end{example}

 \subsection{Orthogonal Grothendieck polynomials}

From now on we specialize to the case when $V = \SymMat_{n\times n}$
and $T$ is the group of invertible $n\times n$ diagonal matrices over our algebraically closed field $\KK$, 
for an arbitrary positive integer $n$.
Then $V$ is a finite-dimensional representation of $T$
under the action $t : M \mapsto tMt$.

As in Section~\ref{schumat-sect3} and Example~\ref{22-k-ex}, we identify
$
\KK[V] = \KK[\SymMat_{n\times n}] = \KK[u_{ij} : 1\leq j\leq i\leq n]
$
and let $a_i : T \to \KK$ for $i \in [n]$ denote the representation $a_i(t) = t_{ii}$.
Then each variable $u_{ij}$ is a weight vector in the $T$-module $\KK[V]$ with weight $a_ia_j$,
and we identify
\[
K_{T}(V)=\ZZ[a_1^{\pm 1},a_2^{\pm 1}, \ldots, a_n^{\pm 1}].
\]

Fix an involution $w \in \I_\infty^{(n)}$. 
Recall the  
variety $\SymX_w$ from Definition~\ref{SymX-def}
and the ideal
$\SymI_w$ 
 defined in
 \eqref{sym-msv-minors-eq}. 
The quotient
$\KK[\SymMat_{n\times n}]/\SymI_w$ is
a finitely generated $(\KK[V], T)$-module,
so we may introduce the following definitions.

\begin{definition} 
For each $w \in \I_\infty^{(n)}$ 
let
$
 \iGconj_w =\cK(\KK[\SymMat_{n\times n}]/\SymI_w) 
$
and
$
 \iG_w=[\SymX_w]_T.
 $ 
 \end{definition}

From the discussion in the previous section and \cite[\S2.4]{MP2020}, we know that  
$
\iGconj_w  $
and 
$
\iG_w $ are both elements of $ \KK[a_1,a_2,\dots,a_n]$.
 The polynomial $\iG_w$ is essentially the \defn{orthogonal Grothendieck polynomial} denoted by the same notation in \cite{MP2020}, except that here we have set a bookkeeping parameter $\beta=-1$. The substitution $\beta=-1$ loses no information
 since the parameter $\beta$ can be reintroduced by a simple change of variables \cite[\S2.2]{MP2020}.
 
  Recall that Conjecture~\ref{sym-conj3a} asserts in part that $\SymI_w = I(\SymX_w)$, which implies that $\iGconj_w = \iG_w$.
 We adopt this conjecture as a hypothesis throughout this section,
 but preserve the distinct notations $\iGconj_w $ and $ \iG_w$ to clarify our arguments.

Given $m \in \PP$ and $w \in S_\infty$,
 define $1^m\times w$ to be the permutation that fixes each $i \in[m]$ and that maps $i+m \mapsto w(i)+m$ for $i\in\PP$.
 We write $1\times w$ in place of $1^1\times w$.
It is evident that if $w \in  \I_\infty^{(n)}$ then $1^m\times w \in  \I_\infty^{(m+n)}$.
 
  Let $x_1,x_2,x_3,\dots$ be the sequence of commuting formal variables with $x_i = 1-a_i$.
 We may then view  $\iG_w$ (as well as $\iGconj_w $) as an element of the polynomial ring
$ \KK[x_1,x_2,\dots,x_n].
$
  It is an open problem \cite[Prob.~5.3]{MP2020} to show that the \defn{stable limit} 
  \be\label{sl-eq}\textstyle \lim_{m \to \infty}\iG_{1^m \times w}\ee converges 
  in the sense of formal power series 
to an element of $\KK\llbracket x_1,x_2,\dots\rrbracket$ for all $w \in  \I_\infty^{(n)}$.

We prove in this section that if Conjecture~\ref{sym-conj3a} holds then this limit does exist,
and in fact gives a \defn{symmetric function}, meaning a power series invariant under all permutations of the $x$-variables.
This convergence is presently only known to hold unconditionally for involutions $w=w^{-1} \in S_n$ that are \defn{vexillary} in the sense
of being $2143$-avoiding \cite{MP2020}.

\begin{example}
Ikeda and Naruse \cite{IkedaNaruse} defined a family of symmetric functions $\GQ_\lambda \in \ZZ\llbracket x_1,x_2,\dots\rrbracket$ indexed by strict integer partitions $\lambda= (\lambda_1>\lambda_2>\dots\geq0)$ to represent the $K$-theory classes of Schubert varieties in the Lagrangian Grassmannian. For a succinct definition, set $\beta=-1$ in the discussion in \cite[\S4.1]{MP2020}.
 When $w=w^{-1} \in S_n$ is vexillary, 
 there is a specific partition $\lambda=\lambda(w)$
 such that 
$
 \lim_{m \to \infty}\iG_{1^m \times w} = \GQ_\lambda
 $ \cite[Thm.~4.11]{MP2020}.

We report a few computations related to the simplest non-vexillary $w \in I_n$.
Assume $\KK$ has characteristic zero.
For $m=1,2,3,4,5$ we have used  {\tt Macaulay2} to calculate that
\[
\ba \iG_{1^{m-1}\times 2143} = 
2\cdot \GQ_{(2)}(x_1,x_2,\dots,x_{m}) 
&- \GQ_{(2,1)}(x_1,x_2,\dots,x_{m}) 
\\&- \GQ_{(3)}(x_1,x_2,\dots,x_{m}) 
\\&+ \GQ_{(3,1)}(x_1,x_2,\dots,x_{m}) 
+ x_{m+1}\cdot A_m + x_{m+2}\cdot B_m
\ea
\]
for certain extra terms $A_m, B_m \in \KK[x_1,x_2,\dots,x_{m+2}]$.
Here $\GQ_{\lambda}(x_1,x_2,\dots,x_{m})$ means the polynomial obtained from $\GQ_\lambda$ by setting $x_{m+1}=x_{m+2}=\dots=0$, which is zero if $m$ is smaller than the number of nonzero parts in $\lambda$.

It would follow that 
$\lim_{m\to \infty}\iG_{1^{m}\times 2143} = 2\cdot \GQ_{(2)} - \GQ_{(2,1)} -\GQ_{(3)} + \GQ_{(3,1)}$
if this pattern were to continue.
(Notice that the sign of the coefficient of each $\GQ_\lambda$ in this expression is the same as $(-1)^{|\lambda|}$. The question of whether this sign alternation is a general phenomenon is open; see \cite[Prob.~5.4]{MP2020}.)
Our methods are only powerful enough to establish convergence of the stable limit \eqref{sl-eq} when Conjecture~\ref{sym-conj3a} holds, however, rather than any exact formulas.
\end{example}

Before getting to our main results, we need to derive several technical lemmas.

\begin{lemma} \label{lem:stable-intersection}
  Let $\iota : \SymMat_{n\times n} \to \SymMat_{(n+1)\times (n+1)}$ be the inclusion $A \mapsto 1 \times A$. Then $\iota^*( \SymI_{1 \times w}) = \SymI_w$ where $\iota^*$ denotes the induced morphism $\KK[\SymMat_{(n+1)\times (n+1)}] \to \KK[\SymMat_{n\times n}]$.
\end{lemma}

\begin{proof} For each $i,j\in[n]$ let $\cM_{ij}(w)$ be the set of minors of size $\rank w_{[i][j]}+1$ inside $\SymcX_{[i][j]}$, so that the ideal $\SymI_w$ is generated by the set
\be\label{cM-eq}\textstyle \cM(w) := \bigcup_{i,j\in[n]}\cM_{ij}(w).\ee The effect of applying $\iota^*$ is to evaluate a polynomial in $\KK[\SymMat_{(n+1)\times (n+1)}]$ on the matrix $1 \times \SymcX_{[n][n]}$. Setting $\tilde{R} = \{i-1 : i \in R \setminus \{1\}\}$ for a set $R \subseteq [n+1]$, we therefore have
    \begin{equation*}
        \iota^*( \det \SymcX_{RC} )= \begin{cases}
            0 & \text{if $1 \in R \mathbin{\triangle} C$ (with $\mathbin{\triangle}$ denoting symmetric difference)}\\
            \det \SymcX_{\tilde{R}\tilde{C}} & \text{otherwise}.
        \end{cases}
    \end{equation*}
Now consider the ideal $\iota^* (\SymI_{1 \times w})$  generated by $\bigcup_{i,j \leq n+1} \iota^* \cM_{ij}(1 \times w)$. We observe that:
\begin{itemize}
    \item If $\min(i,j) = 1$, then $\rank (1 \times w)_{[i][j]}+1 = 2$, so $\cM_{ij}(1 \times w)$ is empty.
    \item If $i,j > 1$, then $\rank (1 \times w)_{[i][j]}+1 = \rank w_{[i-1][j-1]}+2$, so $\iota^* (\cM_{ij}(1 \times w))$ is
    the union of $\cM_{i-1,j-1}(w)$ (coming from the minors involving row and column 1) and certain minors of degree $\rank(w_{[i-1][j-1]})+2$ (coming from the minors not involving row or column 1). However, the latter are in the ideal generated by the former, by the general fact that the ideal of size $r$ minors in a matrix contains the minors of size $r+1$.
\end{itemize}
Thus $\bigcup_{i,j \leq n+1} \iota^*( \cM_{ij}(1\times w))$ and $\cM(w)=\bigcup_{i,j \leq n} \cM_{ij}(w)$ generate the same ideal.
\end{proof}

Below, we distinguish the tori in the algebraic groups $\B_n$ and $\B_{n+1}$ by writing $T_n$ for the former and $T_{n+1}$ for the latter.
Let $t \in T_n$ act on $\SymMat_{(n+1)\times (n+1)}$ by the matrix $\iota(t)$. Then $\iota$ is $T_n$-equivariant, so there is a pullback map $\iota^* : K_{T_n}(\SymMat_{(n+1)\times (n+1)}) \to K_{T_n}(\SymMat_{n\times n})$. The group homomorphism $\iota|_{T_n} : T_n \to T_{n+1}$ also induces a restriction homomorphism 
\[\phi : K_{T_{n+1}}\(\SymMat_{(n+1)\times (n+1)}\) \to K_{T_n}\(\SymMat_{(n+1)\times (n+1)}\),\] which simply sends the $K_{T_{n+1}}$-class of a module to its $K_{T_n}$-class.
\begin{lemma} \label{lem:restriction-action}
  Letting $\iota^*$ and $\phi$ be as above, we have
  \begin{equation*}
    \iota^*(\phi(a_i)) = \begin{cases} 1 & \text{if $i = 1$}\\
      a_{i-1} & \text{otherwise}.
    \end{cases}
  \end{equation*}
\end{lemma}
\begin{proof} By definition, $a_i \in K_{T_{n+1}}(\SymMat_{(n+1)\times (n+1)})$ is $[\KK[\SymMat_{(n+1)\times (n+1)}] \otimes \KK_{a_i}]$ where $\KK_{a_i}$ is the 1-dimensional $T_{n+1}$-representation with character $a_i$. Therefore
  \begin{equation*}
    \phi(a_i) = \Bigl[\Res_{\iota(T_n)}^{T_{n+1}}(\KK[\SymMat_{(n+1)\times (n+1)}] \otimes \KK_{a_i})\Bigr] = \Bigl[\Res_{\iota(T_n)}^{T_{n+1}}(\KK[\SymMat_{(n+1)\times (n+1)}]) \otimes \KK_{a_{i-1}}\Bigr].
  \end{equation*}
  Using the convention that $a_0 = 1$, so that $\KK_{a_0}$ is the trivial representation, this says that 
  \[\phi(a_{i}) = a_{i-1} \in K_{T_n}(\SymMat_{(n+1)\times (n+1)}).\] Now use Lemma~\ref{lem:subrep-pullback}
  to deduce the formula for $ \iota^*(\phi(a_i))$.
\end{proof}

For the next lemma, we make use of the following concepts from \cite[Chapter III]{Hartshorne}.

\begin{definition}[{See \cite[Chapter III, \S9]{Hartshorne}}] Let $C$ be a smooth algebraic curve and $X$ any variety. Let $p_C : X \times C \to C$ and 
    $p_X : X \times C \to X$ be the projections. A collection of subschemes \[\{Z_t \subseteq X : t \in C\}\] is a \defn{flat family} (over $C$) if there is a closed subscheme $Z' \subseteq X \times C$ for which the restriction $p_C|_{Z'}$ is a flat morphism and $Z_t = p_X ( p_C^{-1}(t))$ for all $t \in C$.
\end{definition}

Recall that $G$ is a linear algebraic group.

\begin{lemma} \label{lem:flat-limit} Suppose $X$ is a smooth $G$-variety and $\{Z_t : t \in \CP^1\}$ is a flat family of $G$-stable subschemes over $\CP^1$. Then the classes $[Z_t] \in K_G(X)$ are equal for all $t\in \CP^1$.
\end{lemma}

\begin{proof}
Equipping $C = \CP^1$ with the trivial $G$-action makes $p_C : X \times C \to C$ a $G$-equivariant map. Since $p_C|_{Z'}$ is flat by assumption, we have
\begin{equation*}
    p_C|_{Z'}^*([t]) = [p_C|_{Z'}^{-1}(t)] = [Z_t \times \{t\}] \in K_G(Z').
\end{equation*}
Since the class $[t] \in K_G(\CP^1) = K(\CP^1)$ is independent of $t$, this shows that $[Z_t \times \{t\}] \in K_G(Z')$ is also independent of $t$.

Now consider the commutative diagram 
\begin{equation*}
\begin{CD}
    Z_t \times \{t\}      @>\alpha>>      X \times C\\
    @VV{\tilde{p}_X}V                @VV{p_X}V \\
           X             @>{\id}>>        X
\end{CD}
\end{equation*}  
where $\tilde{p}_X$ is the restriction of $p_X$ to $Z_t \times \{t\}$ and $\alpha$ is inclusion. By \cite[Ch. III, Prop. 9.3]{Hartshorne},
the \defn{higher direct image} functor (as defined \cite[Ch. III, \S8]{Hartshorne}) of the pushforward of $p_X$ satisfies 
\begin{equation*}
    R^i (p_X)_*(\cO_{Z_t \times \{t\}}) = R^i (\tilde{p}_X)_*(\alpha^*\cO_{Z_t \times \{t\}}).
\end{equation*}
Since $\tilde{p}_X$ has 0-dimensional fibers, \cite[Ch. III, Cor. 11.2]{Hartshorne} shows that the right-hand side is 0 for $i > 0$. Therefore the pushforward $(p_X)_* : K_G(X \times C) \to K_G(X)$ sends $[\cO_{Z_t \times \{t\}}]$ to $[(p_X)_* \cO_{Z_t \times \{t\}}]$, i.e., sends the class $[Z_t \times \{t\}]$ to $[p_X(Z_t \times \{t\})] = [Z_t]$. Therefore $[Z_t]$ is also independent of $t$.
\end{proof}

Now suppose $Z \subseteq X \times (C \setminus \{c\})$ is a subscheme where $C$ is a smooth curve and $c \in C$.
Then there is a corresponding  family of subschemes $\{Z_t : t \in C \setminus \{c\}\}$, 
where 
\[ Z_t := \pi_X(Z\cap \pi_C^{-1}(t))\]
 is the fiber of $Z$ over $t$. 
 We define 
$ \lim_{t \to c} Z_t$
  to be the fiber of the closure $\overline{Z} \subseteq X \times C$ over $c$.
  We sometimes denote this limit by
  \be \label{def:limit}
  \textstyle
   Z_c := \lim_{t \to c} Z_t.
   \ee
   It turns out \cite[Ch. III, Prop.~9.8]{Hartshorne} that 
   if $\{Z_t : t \in C \setminus \{c\}\}$ is a flat family over $C \setminus \{c\}$, then the enlarged family $\{Z_t : t \in C\}$ is also flat over $C$.

Define $p : \SymMat_{(n+1)\times (n+1)} \to \SymMat_{n \times n}$ by the formula $p(A) = A_{[2,n+1][2,n+1]}$. This is a left inverse of the map $\iota$ from Lemma~\ref{lem:stable-intersection}.
For $t \in \KK^\times$, define $e_t : \SymMat_{(n+1) \times (n+1)} \to \SymMat_{(n+1) \times (n+1)}$ by 
    \begin{equation*}
    e_t(A)_{ij} = \begin{cases}
    t A_{11} & \text{if $(i,j) = (1,1)$}\\
    A_{ij} & \text{otherwise}.
    \end{cases}
    \end{equation*}
    
\begin{lemma} \label{lem:limit-of-minors}   If $w\in \I_\infty^{(n)}$ then $\lim_{t \to 0} e_t(\SymX_{1 \times w})$ contains $p^{-1}(\SymX_w)$.
\end{lemma}

\begin{proof}
Take an element $A \in p^{-1}(\SymX_w)$ and write $A = \left[\begin{smallmatrix}  a & v^\top \\ v & A' \end{smallmatrix}\right]$ where $a \in \CC$, $v \in \CC^n$, and $A' \in \SymX_w$. First assume $a \neq 0$. Then define $f(t) = at^{-1}$ for $t \neq 0$ and 
\begin{equation*}
     M(t) = \begin{bmatrix} f(t) & v^\top \\ v & A' + f(t)^{-1} vv^\top \end{bmatrix}.
\end{equation*}
We claim that $M(t) \in \SymX_{1 \times w}$. Choose indices $i,j \in [n+1]$. We must check that \[\rank M(t)_{[i][j]} \leq \rank (1\times w)_{[i][j]}.\] This is trivial if $i$ or $j$ is $1$, so assume $i,j > 1$. Then we can write
\begin{equation*}
    M(t)_{[i][j]} = \begin{bmatrix} f(t) & y^\top \\ x & A'_{[i-1][j-1]} + f(t)^{-1} xy^\top \end{bmatrix}
\end{equation*}
where $x \in \KK^{i-1}$ and $y \in\KK^{j-1}$. This matrix is equivalent by row operations to $\left[ \begin{smallmatrix} f(t) & y^\top \\ 0 & A'_{[i-1][j-1]} \end{smallmatrix} \right]$, so
\begin{align*}
\rank  M(t)_{[i][j]}  = \rank \left[ \begin{smallmatrix} f(t) & y^\top \\ 0 & A'_{[i-1][j-1]} \end{smallmatrix} \right]
&\leq 1 + \rank A'_{[i-1][j-1]} 
\\&\leq 1 + \rank w_{[i-1][j-1]} \qquad \text{(since $A' \in \SymX_w$)}\\
&= \rank\, (1\times w)_{[i][j]}.
\end{align*}
Thus $M(t) \in \SymX_{1 \times w}$. Since $\lim_{t\to 0} tf(t) = a$ and $\lim_{t \to 0} f(t)^{-1} = 0$, we have $\lim_{t \to 0} e_t M(t) = A$. 
Since we assumed $a\neq 0$, this shows that $\lim_{t \to 0} e_t \SymX_{1 \times w}$ contains $\{A \in p^{-1}(\SymX_{w}) : A_{11} \neq 0\}$. But   $\lim_{t \to 0} e_t \SymX_{1 \times w}$  also contains the latter set's closure, which is $p^{-1}(\SymX_{w})$ as needed.
\end{proof}

We mention a basic fact for use in the proof of the next lemma.
\begin{proposition} \label{prop:union-subset}
Suppose ideals $I,J,K$ satisfy $J \cap K \subseteq I \subseteq K$. Then $I = (I + J) \cap K$.
\end{proposition}

\begin{proof}
The inclusion $I \subseteq (I+J) \cap K$ is clear. For the other, suppose $x = i + j = k$ with $i \in I$, $j \in J$, $k \in K$. Then $j = k-i \in K$, hence $j \in J \cap K \subseteq I$, and so $x = i+j \in I$.
\end{proof}

\begin{lemma} \label{lem:stab} 
Suppose Conjecture~\ref{sym-conj3a} holds 
and $w \in \I_\infty^{(n)}$. Then $\iG_w$ is obtained from the polynomial $\iG_{1\times w} \in \KK[a_1,a_2,\dots,a_{n+1}]$ by substituting $a_1\mapsto 0$ and $a_{i+1}\mapsto a_i$ for all $i \in [n]$.
\end{lemma}

\begin{proof}
Recall the definition of $\cM(w)$ from \eqref{cM-eq}.
For $t \neq 0$, let $V_t$ be the variety $e_t(\SymX_{1\times w})$, and set $V_0 = \lim_{t \to 0} V_t$. The ideal of $V_t$ is $(e_t^*)^{-1} I(\SymX_{1 \times w})$, which, assuming Conjecture~\ref{sym-conj3a}, equals $(e_t^*)^{-1} \SymI_{1 \times w}$. The latter is generated by the minors $\cM(1 \times w)$ except that all monomials \emph{not} divisible by $u_{11}$ get multiplied by $t$. 

Setting $t = 0$, we see that $I(V_0)$ contains $u_{11}p^*\cM(w)$, which generate the ideal $\langle u_{11} \rangle \cap \langle p^*(\SymI_w) \rangle$ since the elements of $p^*\cM(w)$ lie in the subring generated by $\{u_{ij} : i,j > 1\}$.  On the other hand, $V_0$ contains $p^{-1}(\SymX_w)$ by Lemma~\ref{lem:limit-of-minors}, so $I(V_0)$ contains $I(p^{-1}(\SymX_w)) = \langle p^*(\SymI_w) \rangle$, again using Conjecture~\ref{sym-conj3a}. By Proposition~\ref{prop:union-subset}, $I(V_0) = (I(V_0)+\langle u_{11}\rangle) \cap \langle p^*(\SymI_w) \rangle$. Equivalently, $V_0$ is the scheme-theoretic union $Z \cup p^{-1}(\SymX_w)$ for some closed subscheme $Z$ of $\{u_{11} = 0\}$.

Now consider the $K$-class 
\begin{equation*}
[V_0] = [Z] + p^*[\SymX_w] - [Z \cap p^{-1}(\SymX_w)].
\end{equation*}
Any closed subscheme of $\{u_{11} = 0\} \subseteq \SymMat_{(n+1)\times (n+1)}$ has zero $K_{T_n}$-class as $T_n$ acts trivially on $u_{11}$. Therefore, the first and ,ast terms on the right side of the preceding displayed equation vanish, so
\[
[V_0] = p^*[\SymX_w].\]
The subschemes $e_t V(\SymI_{1 \times w})$ are all isomorphic for $t \neq 0$, so they form a flat family over $\mathbb{A}^1 \setminus \{0\}$ by \cite[Ch. III, Thm.~9.9]{Hartshorne}. The discussion around \eqref{def:limit} and Lemma~\ref{lem:flat-limit} therefore imply that
\[
[\SymX_{1 \times w}] = [V_0] = p^*[\SymX_w].\] As $p \circ \iota = \id$, we conclude that $\iota^*[\SymX_{1 \times w}] = [\SymX_w]$. By Lemma~\ref{lem:restriction-action} this completes the proof.
\end{proof}

\begin{lemma} \label{lem:sym} If $w \in \I_\infty^{(n)}$ then the polynomial $\iGconj_{1^m \times w}\in \KK[a_1,a_2,\dots,a_{m+n}]$ is symmetric in the variables $a_1,a_2, \ldots, a_m$. \end{lemma}

\begin{proof}
    Identify $\sigma \in S_m$ with the permutation matrix $\sigma \times 1^n$, so that $S_m$ acts on $\SymMat_{m+n\times m+n}$ by conjugation. The induced map on the coordinate ring $\KK[\SymMat_{m+n\times m+n}]$ sends $u_{ij}$ to $u_{\sigma(i),\sigma(j)}$. For each $i,j$, let $\cM_{ij}$ be the set of minors of size $\rank(1^m \times w)_{[i][j]}+1$ inside $\SymcX_{[i][j]}$. Recall that $\SymI_{1^m \times w}$ is generated by $\bigcup_{i,j} \cM_{ij}$. Now, we observe the following properties:
    \begin{itemize}
        \item If $\min(i,j) \leq m$, then $\rank({1^m \times w})_{[i][j]}+1 = \min(i,j)+1$, so $\cM_{ij}$ is empty.
        \item If $\sigma$ is a permutation fixing every $k > \max(i,j)$, then $\sigma$ maps $\cM_{ij}$ to itself; in particular, this holds if $\sigma \in S_m$ and $i,j > m$.
    \end{itemize}
    It follows that $\sigma$ maps $\bigcup_{i,j} \cM_{ij}$ to itself, and hence also maps $\SymI_{1^m \times w}$ to itself. The action of $\sigma$ on $\SymMat_{m+n\times m+n}$ induces an action on $K_{T_{m+n}}(\SymMat_{m+n\times m+n})$ permuting the variables $a_1,\dots,a_m$, and $\sigma \SymI_{1^m \times w} = \SymI_{1^m \times w}$ implies that $\iGconj_{1^m \times w} \in \KK[a_1,a_2,\dots,a_{m+n}]$ is also fixed by $\sigma$.
\end{proof}

Recall that  $x_i = 1-a_i$ and that we consider $\iG_w$ as a polynomial in $x_i$.

\begin{theorem} \label{thm:stability}  
Assume that Conjecture~\ref{sym-conj3a} holds. Choose any element $w \in \I_\infty^{(n)}$. Then
the stable limit $\lim_{m \to \infty} \iG_{1^m \times w}$ converges to a symmetric element of $\KK\llbracket x_1,x_2,x_3,\dots,\rrbracket$.
\end{theorem}

\begin{proof}
Write $\iG_{1 \times w} = \iG_{1 \times w}(x_1,x_2,\ldots,x_{n+1})$ and $\iG_w =\iG_w(x_1,x_2,\ldots,x_n)$ as polynomials in the $x$-variables.
Then 
     Lemma~\ref{lem:stab} shows that $\iG_{1 \times w}(0, x_1, \ldots, x_n) = \iG_w(x_1, \ldots, x_n)$. Now if $x_1^{\alpha_1} \cdots x_d^{\alpha_d}$ is some monomial with $d < m$ then 
     \begin{align*}
        [x_1^{\alpha_1} \cdots x_d^{\alpha_d}] \iG_{1^m \times w}(x_1, \ldots, x_{m+n}) &= [x_1^{\alpha_1} \cdots x_d^{\alpha_d}] \iG_{1^m \times w}(x_1, \ldots, x_{m-1}, 0, \ldots, 0)\\
        &= [x_1^{\alpha_1} \cdots x_d^{\alpha_d}] \iG_{1^m \times w}(0, x_1, \ldots, x_{m-1}, 0, 0, \ldots, 0)\\
        &= [x_1^{\alpha_1} \cdots x_d^{\alpha_d}] \iG_{1^{m-1} \times w}(x_1, \ldots, x_{m-1}, 0, 0, \ldots, 0),
     \end{align*}
     using Lemma~\ref{lem:sym} for the second equality (plus the assumption that $\iG_w = \iGconj_w$). By induction 
     \begin{equation*}
        \lim_{m \to \infty} [x_1^{\alpha_1} \cdots x_d^{\alpha_d}] \iG_{1^m \times w} = [x_1^{\alpha_1} \cdots x_d^{\alpha_d}]\iG_{1^d \times w}(x_1, \ldots, x_d, 0, \ldots, 0).
     \end{equation*}
Thus the sequence of coefficients of any fixed monomial in $\iG_{1^m \times w}$ is eventually constant as $m\to \infty$, so  $\lim_{m \to \infty} \iG_{1^m \times w}$ converges to a formal power series, which is symmetric
by Lemma~\ref{lem:sym}.
\end{proof}

\end{document}